\documentclass[12pt, reqno]{amsart}
\usepackage[english]{babel}
\usepackage[utf8]{inputenc}
\usepackage{lipsum}
\usepackage{dsfont}
\usepackage{mathrsfs}
\usepackage{stix}
\usepackage{fullpage}
\usepackage{mathtools}
\usepackage{amsmath}
\usepackage{bbm}
\usepackage{amsfonts}
\usepackage{tikz}
\usepackage{tikz-cd}
\usetikzlibrary{decorations.pathreplacing,angles,quotes}
\usetikzlibrary{arrows,chains,positioning,scopes,quotes}
\newtheorem{defi}{Definition}[section]

\newtheorem*{conjecture}{Conjecture}
\newtheorem{lema}[defi]{Lemma}
\newtheorem{teo}[defi]{Theorem}
\newtheorem{rem}[defi]{Remark}

\newtheorem*{rem*}{Remark}

\newcommand{\C}{\mathbb{C}}
\newcommand{\Q}{\mathbb{Q}}
\newcommand{\K}{\mathbb{K}}
\newcommand{\R}{\mathbb{R}}
\newcommand{\N}{\mathbb{N}}
\newcommand{\Z}{\mathbb{Z}}

\newcommand{\1}{\mathds{1}}

\DeclareMathOperator{\spn}{span}

\newcommand{\interior}[1]{%
  {\kern0pt#1}^{\mathrm{o}}%
}

\usepackage[colorlinks=false]{hyperref}
\renewcommand\eqref[1]{(\ref{#1})} 
\begin{document}

\title[The spectrum of the Vladimirov sub-Laplacian on the compact Heisenberg group ]
 {The spectrum of the Vladimirov sub-Laplacian on the compact Heisenberg group }

\author{
  J.P. Velasquez-Rodriguez
}

\newcommand{\Addresses}{{
  \bigskip
  \footnotesize

 J.P. Velasquez-Rodriguez, \textsc{Department of Mathematics: Analysis, Logic and Discrete Mathematics, Ghent University, Belgium, and Departamento de Matematicas, Universidad del Valle, Cali-Colombia}\par\nopagebreak
  \textit{E-mail address:} \texttt{juanpablo.velasquezrodriguez@ugent.be / velasquez.juan@correounivalle.edu.co}}

}


\thanks{The author is supported by the grace of God and Lord Bruce the First, ruler of the Lands Between}



\date{January 15, 2024.}
\begin{abstract}
Let $p>2$ be a prime number. In this short note, we calculate explicitly the unitary dual and the matrix coefficients of the Heisenberg group over the $p$-adic integers. As an application, we consider directional Vladimirov–Taibleson derivatives, and some polynomials in these operators. In particular, we calculate explicitly the spectrum of the Vladimirov sub-Laplacian and show how it provides a non-trivial example of a globally hypoelliptic operator on compact nilpotent $\K$-Lie groups.
\end{abstract}
\maketitle
\tableofcontents
\section{Introduction}
This paper is motivated by a principle that Harish-Chandra referred to as the ``Lefschetz principle'', which says that \emph{real
groups, $p$-adic groups and automorphic forms (corresponding to the archimedean and non-archimedean local fields, and to number fields) should be placed on an equal footing, and that ideas and results from one of these three categories should transfer to the other two} \cite{HarisLanglands}. In this context, we focus on extending certain techniques from harmonic analysis on real groups to compact $p$-adic Lie groups, a class of topological groups that has primarily been studied in the literature from an algebraic perspective. Our main concern will be Fourier analysis and the study of the Vladimirov-Taibleson operator, which parallels the fractional Laplacian in real analysis, in the context of non-commutative groups over the $p$-adic integers.  

\newpage
Let us elucidate the meaning of the phrase \emph{the extension of some techniques from Harmonic analysis on Lie groups}. This is important because the analysis on Lie groups relies on the differential structure of the group as a smooth manifold. So far, in the knowledge of the author, there is no theory of derivative, differential, or pseudo-differential operators on $p$-adic manifolds, and neither on the class of non-commutative locally profinite groups, even though there is a rich theory on $p$-adic Lie groups, algebraic groups, and $\K$-analytic manifolds. Only in recent times P.E. Bradley and collaborators started the study of pseudo-differential operators invariant under a finite group actions, heat equations on Mumford curves, and $p$-adic Laplacians \cite{Bradley3, Bradley2, bradley1}. And apart from this, to the best knowledge of the author, only Kochubei has considered equations on a $p$-adic ball \cite{Kochubei2018}, and there are some works about general ultrametric spaces \cite{Bendikov2014}. As it turns out, the common thread on the $p$-adic side of the theory is the Vladimirov–Taibleson operator, which provides a definition of ``differentiability'' for functions defined on ultrametric spaces. Most pseudo-differential equations in the literature, specially those aimed to have some real world application, are given in terms of this operator, or a similar one \cite{Khrennikov2018}, and it is commonly regarded as some kind of fractional Laplacian for complex valued functions on totally disconnected spaces \cite{Kochubei2023}, specially because, in some cases, it actually coincides with the isotropic Laplacian associated to a certain ultrametric in general ultrametric spaces, as the works of A. Bendikov, A. Grigoryan, C. Pittet and W. Woess showed. In particular, when the ultrametric space has additionally some group structure, the Vladimirov-Taibleson operator can be very nicely expressed in terms of the group’s Fourier analysis. 

 In the real case, given a partial differential operator $L$, a fundamental question in the theory of PDEs is: \emph{If $Lu = f$ and $f$ is smooth, does this imply that $u$ is smooth?} This question leads to the concept of \emph{hypoellipticity}, which is a central topic in the theory of pseudo-differential operators on smooth manifolds. In this paper, our primary focus is on $\mathbb{H}_d(\mathbb{Z}_p)$, or simply $\mathbb{H}_d$ for short, which we refer to as the $(2d+1)$-dimensional Heisenberg group over the $p$-adic integers. We define it as $\mathbb{Z}_p^{2d + 1}$ equipped with the non-commutative operation 
\[
(\mathbf{x}, \mathbf{y}, z) \star (\mathbf{x}', \mathbf{y}', z') := \left(\mathbf{x} + \mathbf{x}', \mathbf{y} + \mathbf{y}', z + z' + \mathbf{x} \cdot \mathbf{y}'\right),
\]which endows it with an analytic manifold structure over $\mathbb{Q}_p$ along with a compatible group structure, making it a (compact) $p$-adic Lie group. For this $p$-adic manifold, we aim to study the problem of global hypoellipticity using strategies analogous to those employed for compact Lie groups. The idea is straightforward: when the manifold possesses an additional group structure, we can employ the group Fourier analysis to simplify the study of linear operators, particularly those who are translation invariant. This approach proves to be very useful for investigating the problem of global hypoellipticity on compact Lie groups, as it is shown in \cite{Kirilov2020}, and we will use it here for the totally disconnected case.

In general, proving that a given partial differential operator is hypoelliptic is a non-trivial problem, and there is a substantial body of literature dedicated to addressing this question. Fortunately, there are cases where a satisfactory answer exists, as in the case of \emph{H{\"o}rmander's operators}. See \cite{Bramanti2014} and the references therein. In 1967, H{\"o}rmander proved that, under the assumption that the system of smooth vector fields $X_1, \ldots, X_\kappa$ generates the entire tangent space at any point, a condition nowadays called \emph{H{\"o}rmander's condition}, the operator
\begin{equation}
    L := \sum_{i=1}^\kappa X_i^2 + X_0 + c,
\end{equation}
is a hypoelliptic operator \cite{Hrmander1967}. An operator in this form is commonly known as a \emph{H{\"o}rmander's operator}, and by taking $X_0 =0$ and $c=0$ we get a \emph{H{\"o}rmander's} sum of squares. Important examples of H{\"o}rmander's operators are the sub-Laplacians on Lie groups \cite{Bramanti2014}. 

In this paper we consider the problem of the global hypoellipticity on compact nilpotent groups over non-archimedean local fields. Specifically, we aim to introduce the definition of \emph{directional Vladimirov–Taibleson operators} along with the polynomials associated with these operators. The Vladimirov–Taibleson operator offers a concept of differentiability for functions defined on profinite groups, and the directional VT operators bear some resemblance to directional derivatives. 

\begin{defi}\label{defiDirectionalK}
Let $\K$ be a non-archimedean local field with ring of integers $\mathscr{O}_\K$, prime ideal $\mathfrak{p}= \textbf{p} \mathscr{O}_\K$, and residue field $\mathbb{F}_q = \mathscr{O}_\K/\textbf{p} \mathscr{O}_\K.$ Let $\mathfrak{g} = \spn_{\mathscr{O}_\K} \{ X_1,..,X_d \}$ be a nilpotent $\mathscr{O}_\K$-Lie algebra, and let $\mathbb{G}$ be the exponential image of $\mathfrak{g}$, so that $\mathbb{G}$ is a compact nilpotent $\K$-Lie group. We will use the symbol symbol $\partial_{X}^\alpha$ to denote the \emph{directional Vladimirov–Taibleson operator in the direction of $X \in \mathfrak{g}$}, or directional VT operator for short, which we define as $$\partial_{X}^\alpha f(\mathbf{x}) := \frac{1 - q^\alpha}{1-q^{-(\alpha +1)}} \int_{\mathscr{O}_\K} \frac{f(\mathbf{x} \star \mathbf{exp}(tX)^{-1}) - f(\mathbf{x}) dt}{|t|_\K^{\alpha +1}}, \, \, \, \, f \in \mathcal{D}(\mathbb{G}).$$Here $\mathcal{D}(\mathbb{G})$ denotes the space of smooth functions on $\mathbb{G}$, i.e., the collection of locally constant functions with a fixed index of local constancy.   
\end{defi}

\begin{rem}
Notice how the above definition is independent of any coordinate system we choose for the group as a manifold. Even though we are using the exponential map in the definition, this does not mean we are using the exponential system of coordinates. First, $\mathbf{x} \in \mathbb{G}$ is not written in coordinates, and second, we know there is a one to one correspondence between elements of the Lie algebra $\mathfrak{g}$ and one-parameter subgroups of $\mathbb{G}$. Here a one-parameter subgroup of $\mathbb{G}$ is an analytic homomorphism $\gamma_X:\mathscr{O}_\K \to \mathbb{G}$.  Using this fact, given any $X \in \mathfrak{g}$, we can define the directional VT operator in the direction of $X \in \mathfrak{g}$ via the formula  $$\partial_{X}^\alpha f(\mathbf{x}) := \frac{1 - q^\alpha}{1-q^{-(\alpha +1)}} \int_{\mathscr{O}_\K} \frac{ f(\mathbf{x} \star \gamma_X(t)^{-1}) - f(\mathbf{x}) dt}{|t|_\K^{\alpha +1}}.$$
Since $\mathbb{G}$ is compact, we can identify it with a matrix group, where all one parameter subgroups have the form $\gamma_X (t) = e^{tX}$, which justifies our initial definition. More generally, we can write any analytic vector field on $\mathbb{G}$ as $$X(\mathbf{x}):= 
\sum_{j=1}^d c_j (\mathbf{x}) X_j,$$where the coefficient functions $c_j$ are analytic. For this vector field, the directional VT operator is defined in a similar way as $$\partial_{X}^\alpha f(\mathbf{x}) := \frac{1 - q^\alpha}{1-q^{-(\alpha +1)}} \int_{\mathscr{O}_\K} \frac{f(\mathbf{x} \star \mathbf{exp}(tX (\mathbf{x}))^{-1}) - f(\mathbf{x}) dt}{|t|_\K^{\alpha +1}}.$$However, in this work we wont consider the case where the coefficients are not constant, because such operators are not necessarily invariant.         
\end{rem}

By using Definition \ref{defcompactdirectionalVT}, we can link an specific Vladimirov-type operator to each direction $X \in \mathfrak{g}$ and subsequently study the resulting operators. While this association does not preserve the Lie algebra structure, as seen in Lie groups over the real numbers, the resulting operators are nonetheless interesting and share similarities with partial differential operators on Lie groups. We want to study polynomials in these operators and, specifically, a distinguished operator called here the \emph{Vladimirov sub-Laplacian}. For this operator, which we define in principle for compact nilpotent groups, we would like to pose the following conjecture:

\begin{conjecture}\normalfont   
 Let $\K$ be a non-archimedean local field with ring of integers $\mathscr{O}_\K$, prime ideal $\mathfrak{p}= \textbf{p} \mathscr{O}_\K$, and residue field $\mathbb{F}_q = \mathscr{O}_\K/\textbf{p} \mathscr{O}_\K.$ Let $\mathfrak{g} = \spn_{\mathscr{O}_\K} \{ X_1,..,X_d \}$ be a nilpotent $\mathscr{O}_\K$-Lie algebra, and let $\mathbb{G}$ be the exponential image of $\mathfrak{g}$, so that $\mathbb{G}$ is a compact nilpotent $\K$-Lie group. Let $X_1,...,X_{\kappa}$, $1 \leq \kappa \leq d$, be an $\mathscr{O}_\K$-basis for $\mathfrak{g}/[\mathfrak{g}, \mathfrak{g}]$, and let $\mathbb{G}$ be the exponential image of $\mathfrak{g}$. Then the \emph{Vladimirov sub-Laplacian of order $\alpha>0$} is a hypoelliptic operator on $\mathbb{G}$. Here the Vladimirov sub-Laplacian is the pseudo-differential operator $\mathscr{L}^\alpha_{sub}$, defined on the space of smooth functions $\mathcal{D}(\mathbb{G})$ via the formula $$\mathscr{L}^\alpha_{sub} f(\mathbf{x}) := \sum_{k= 1}^\kappa 
\partial_{X_k}^\alpha f(\mathbf{x}),$$where $X_1 , ...,X_\kappa $ spans $\mathfrak{g}/[\mathfrak{g} , \mathfrak{g}]$ as $\mathscr{O}_\K$-module. 
\end{conjecture}

\begin{rem}\normalfont
    To the knowledge of the author, an operator like the Vladimirov sub-Laplacian has not appeared before in the mathematical literature. The closest thing is probably the study of the \emph{Vladimirov-Laplacian} by Bendikov,  Grigoryan,  Pittet,  and Woess \cite{Bendikov2014}. There the authors studied the operators $$\mathscr{L}^\alpha f(x) := \sum_{k= 1}^d 
\partial_{X_k}^\alpha f(x),$$as the generators of a Dirichlet form associated with a certain jump kernel. See \cite[Section 5]{Bendikov2014} for more details. 
\end{rem}

So, drawing an analogy between the real case and the non-archimedean case, we can think on the Vladimirov sub-Laplacian as an analog of the sub-Laplacian or H{\"o}rmander's sum of squares for the non-archimedean case. Consequently, it is reasonable to expect the global hypoellipticity for these operators, and the main goal of this paper is to demonstrate that this is indeed the case for the Heisenberg group over the $p$-adic integers, denoted  by $\mathbb{H}_d(\Z_p)$.

As the main results of this paper, for the particular case $\K=\Q_p$ and $\mathbb{G}=\mathbb{H}_d (\Z_p)$, we will provide a complete description of the unitary dual, the matrix coefficients of the representations, and the spectrum of the Vladimirov sub-Laplacian on $\mathbb{H}_d$. To be more precise, our goal in this work is to establish the following two results, which justify our conjecture about sub-Laplacians on more general nilpotent groups.  
 \begin{rem}
    In this paper, we will identify each equivalence class $\lambda$ in $\widehat{\Z}_p \cong \Q_p / \Z_p$, with its associated representative in the complete system of representatives $$\{1\} \cup \big\{ \sum_{k =1}^\infty \lambda_k p^{-k} \, : \, \, \text{only finitely many $\lambda_k$ are non-zero.} \big\}.$$Similarly, every time we consider an element of the quotients $\Q_p / p^{-n} \Z_p$ it will be chosen from the complete system of representatives   $$\{1\} \cup \big\{ \sum_{k =n+1}^\infty \lambda_k p^{-k} \, : \, \text{only finitely many $\lambda_k$ are non-zero} \big\}.$$Also, for any $p$-adic number we will write $\vartheta(\lambda)$ for the $p$-adic valuation of $\lambda \in \Q_p$.
\end{rem}
Using the conventions introduced above, we have the following explicit description of the unitary dual of $\mathbb{H}_d$.

\begin{teo}\label{TeoRepresentationsHd}
Let us denote by $\widehat{\mathbb{H}}_d$ the unitary dual of $\mathbb{H}_d$, i.e., the collection of all equivalence classes of unitary irreducible representations of $\mathbb{H}_d$. Then we can identify $\widehat{\mathbb{H}}_d$ with the following subset of $\widehat{\Z}_p^{2d+1} \cong \Q_p^{2d+1}/\Z_p^{2d+1}$: $$\widehat{\mathbb{H}}_d = \{(\xi , \eta, \lambda ) \in  \widehat{\Z}_p^{2d+1} \, : (\xi , \eta) \in \Q_p^{2d} / p^{\vartheta(\lambda)} \Z_p^{2d} \}.$$Moreover, each non-trivial representation $[\pi_{(\xi, \eta, \lambda)}] \in \widehat{\mathbb{H}}_d $ can be realized in the finite dimensional sub-space $\mathcal{H}_\lambda$ of $L^2(\Z_p^d)$ defined as
$$\mathcal{H}_\lambda := \spn_\C \{ \varphi_h \, : \, h \in \Z_p^d / p^{-\vartheta(\lambda)}\Z_p^d  \}, \, \, \, \varphi_h (u) := |\lambda|_p^{d/2} \1_{h + p^{-\vartheta(\lambda)} \Z_p^d} (u), \, \, \dim(\mathcal{H}_\lambda) =|\lambda|_p^d,$$
where the representation acts on functions $\varphi \in \mathcal{H}_\lambda$ according to the formula 
$$\pi_{(\xi , \eta , \lambda)}(\mathbf{x}, \mathbf{y}, z) \varphi (u) := e^{2 \pi i \{\xi \cdot \mathbf{x} + \eta \cdot \mathbf{y} + \lambda (z +  u \cdot \mathbf{y}) \}_p} \varphi (u + \mathbf{x}), \, \, \, \, \varphi \in \mathcal{H}_\lambda.$$With this explicit realization, and by choosing the basis $\{ \varphi_h\,: \, h \in \Z_p^d / p^{-\vartheta(\lambda)}\Z_p^d\}$ for each representation space, the associated matrix coefficients are given by $$(\pi_{(\xi , \eta , \lambda)})_{h h'}(\mathbf{x}, \mathbf{y},z)=e^{ 2 \pi i \{ (\xi \cdot \mathbf{x} + \eta \cdot \mathbf{y}) +  \lambda(z + h' \cdot \mathbf{y} )  \}_p} \1_{h' - h + p^{-\vartheta(\lambda)}  \Z_p^d}(\mathbf{x}) ,$$and the corresponding characters are $$\chi_{\pi_{(\xi, \eta, \lambda)}} (\mathbf{x}, \mathbf{y}, z) = e^{2 \pi i \{ \xi \cdot \mathbf{x} + \eta \cdot \mathbf{y} + \lambda z \}_p} \1_{p^{-\vartheta(\lambda)} \Z_p^d} (\mathbf{x} ) \1_{p^{-\vartheta(\lambda)} \Z_p^d}(\mathbf{y}).$$Sometimes we will use the notation $$\mathcal{V}_{(\xi , \eta , \lambda)}:= Span_\C \{ (\pi_{(\xi , \eta , \lambda)})_{h h'} \, : \, h,h' \in \Z_p^d / p^{-\vartheta(\lambda)} \Z_p^d \}.$$
\end{teo}

\begin{rem}\label{remrepres}
The representation theory of a group like $\mathbb{H}_d$ has been considered before in \cite{Boyarchenko2008, Howe1977KirillovTF} where more general classes of groups are studied. However, in the knowledge of the author, there is no general method available in the literature to obtain explicitly the matrix coefficients of the representations, even for very simple cases like the Heisenberg group.  
\end{rem}

Once we have the description of the unitary dual given in Theorem \ref{TeoRepresentationsHd}, we can prove the following spectral theorem for the Vladimirov sub-Laplacian on $\mathbb{H}_d$:

\begin{teo}\label{TeoSpectrumSublaplacianHd}
Let $\mathfrak{h}_d$ be the $(2d+1)$-dimensional Heisenberg Lie algebra, with generators $$\{ X_1,...,X_d,Y_1,...,Y_d,Z\}, \, \, \, [X_i, Y_j] = \delta_{ij} Z,$$ and let $\textbf{V} = \{V_1,...,V_d \} \subset Span_{\Z_p} \{ X_1,...,X_d\}, \,\, \textbf{W} =  \{W_1,...,W_d \} \subset Span_{\Z_p} \{ Y_1,...,Y_d\}$ be collections of linearly independent vectors. The Vladimirov sub-Laplacian associated to this collection  $$T_{\textbf{V}, \textbf{W}}^{\alpha } : = \sum_{k=1}^\varkappa \partial_{V_k}^{\alpha} + \partial_{W_k}^{\alpha} ,$$defines a left-invariant, self-adjoint operator on $\mathbb{H}_d$, which is also globally hypoelliptic. The spectrum of this operator is purely punctual, and its associated eigenfunctions form an orthonormal basis of $L^2(\mathbb{H}_d)$. Furthermore, the symbol of $T_{\textbf{V}, \textbf{W}}^{\alpha }$ acts on each representation space as a $p$-adic Schr{\"o}dinger operator, and the space $L^2(\mathbb{H}_d)$ can be written as the direct sum $$L^2(\mathbb{H}_d) = \overline{\bigoplus_{(\xi,\eta,\lambda) \in \widehat{\mathbb{H}}_d} \bigoplus_{h' \in  \Z_p^d / p^{-\vartheta(\lambda)} \Z_p^d} \mathcal{V}_{(\xi , \eta, \lambda)}^{h'}}, \, \,\, \mathcal{V}_{(\xi , \eta , \lambda)} = \bigoplus_{h' \in  \Z_p^d / p^{-\vartheta(\lambda)} \Z_p^d} \mathcal{V}_{(\xi , \eta, \lambda)}^{h'}, $$where each finite-dimensional sub-space$$\mathcal{V}_{(\xi , \eta, \lambda)}^{h'}:= \mathrm{span}_\C \{ (\pi_{(\xi , \eta , \lambda)})_{hh'} \, : \, h \in \Z_p^d / p^{-\vartheta(\lambda)} \Z_p^d \},$$is an invariant sub-space where $T_{\textbf{V}, \textbf{W}}^{\alpha }$ acts like the Schr{\"o}dinger-type operator operator $$ \sum_{k=1}^d \partial_{V_k}^{\alpha} + |W_k \cdot ( \lambda  h' + \eta ) |_p^{\alpha} - \frac{1 - p^{-1}}{1 - p^{-(\alpha +1)}}.$$
Consequently, the spectrum of $T_{\textbf{V}, \textbf{W}}^{\alpha }$ restricted to $\mathcal{V}_{(\xi , \eta, \lambda)}^{h'}$ is given by $$Spec(T_{\textbf{V}, \textbf{W}}^{\alpha }|_{\mathcal{V}_{(\xi , \eta, \lambda)}^{h}})= \Big\{ \sum_{k=1}^d |V_k \cdot(\tau + \xi)|_p^\alpha + |W_k \cdot ( \lambda  h' + \eta ) |_p^{\alpha} - 2\frac{1 - p^{-1}}{1 - p^{-(\alpha +1)}} \, : \, 1 \leq\| \tau \|_p \leq |\lambda|_p \Big\},$$
so that $\mathrm{Spec}(T_{\textbf{V}, \textbf{W}}^{\alpha })$ is going to be the collection of real numbers \begin{align*}
    \sum_{k=1}^d |V_k \cdot(\tau + \xi)|_p^\alpha + |W_k \cdot ( \lambda  h' + \eta ) |_p^{\alpha} - 2\frac{1 - p^{-1}}{1 - p^{-(\alpha +1)}},
\end{align*}where $(\xi, \eta, \lambda) \in \widehat{\mathbb{H}}_d, \, h' \in \Z_p^d / p^{-\vartheta(\lambda)} \Z_p^d , \,  1 \leq \| \tau \|_p \leq |\lambda|_p$, and the corresponding eigenfunctions are given by  $$e^{ 2 \pi i \{ \lambda(z +  h' \cdot \mathbf{y} ) + ( (\xi+\tau) \cdot \mathbf{x} +  \eta \cdot \mathbf{y}) \}_p}, \, \, \, \,(\xi, \eta, \lambda) \in \widehat{\mathbb{H}}_d, \, h' \in \Z_p^d / p^{-\vartheta(\lambda)} \Z_p^d , \,  1 \leq \| \tau \|_p \leq |\lambda|_p.$$
\end{teo}
\begin{rem}
    Here a Schr{\"o}dinger-type operator is simply an operator of the form $D^\alpha + V$, where $D^\alpha$ is some Vladimirov-type operator and $V$ is a potential.   
\end{rem}

\section{Preliminaries}
\subsection{The field of $p$-adic numbers $\Q_p$}
Throughout this article $p>2$ will always denote a fixed prime number. The field of $p$-adic numbers, usually denoted by $\Q_p$, can be defined as the completion of the field of rational numbers $\Q$ with respect to the $p$-adic norm $|\cdot|_p$, defined as \[|u|_p := \begin{cases}
0 & \, \text{if} \, u=0, \\ p^{-l} & \, \text{if} \, u= p^{l} \frac{a}{b},
\end{cases}\]where $a$ and $b$ are integers coprime with $p$. The integer $l:= \vartheta(u)$, with $\vartheta(0) := + \infty$, is called the \emph{$p$-adic valuation} of $u$. The unit ball of $\Q_p^d$ with the $p$-adic norm $$\|u \|_p:=\max_{1 \leq j \leq d} |u_j|_p,$$ is called the group of $p$-adic integers, and it will be denoted by $\Z_p^d$. Any $p$-adic number $u \neq 0$ has a unique expansion of the form $$u = p^{\vartheta(u)} \sum_{j=0}^{\infty} u_j p^j,$$where $u_j \in \{0,1,...,p-1\}$ and $u_0 \neq 0$. With this expansion we define the fractional part of $u \in \Q_p$, denoted by $\{u\}_p$, as the rational number\[\{u\}_p := \begin{cases}
0 & \, \text{if} \, u=0 \, \text{or} \, \vartheta(u) \geq 0, \\ p^{\vartheta(u)} \sum_{j=0}^{-\vartheta(u)-1} u_j p^j,& \, \text{if} \, \vartheta(u) <0.
\end{cases}\] $\Z_p^d$ is  compact, totally disconected, i.e. profinite, and abelian. Its dual group in the sense of Pontryagin, the collection of characters of $\Z_p^d$, will be denoted by $\widehat{\Z}_p^d$. The dual group of the $p$-adic integers is known to be the Pr{\"u}fer group $\Z (p^{\infty})$,  the unique $p$-group in which every element has $p$ different $p$-th roots. The Pr{\"u}fer group may be identified with the quotient group $\Q_p/\Z_p$. In this way, the characters of the group $\Z_p^d$ may be written as $$\chi_p (\tau  u) := e^{2 \pi i \{\tau \cdot u \}_p}, \, \, u \in \Z_p^d, \, \tau \in \widehat{\Z}_p^d 
\cong \Q_p^d / \Z_p^d .$$

\begin{rem}
    Along this work, it will be convenient at several points to identify the elements of $\widehat{\Z}_p$, which are equivalence classes $[\tau] \in \Q_p / \Z_p$, with the representative of the class $\tau = \sum_{k=1}^\infty \tau_k p^{-k},$ where only finitely many $\tau_k \in \mathbb{F}_p$ are different from zero.   
\end{rem}

By the Peter–Weyl theorem, the elements of $\widehat{\Z}_p^d$ constitute an orthonormal basis for the Hilbert space $L^2 (\Z_p^d)$, which provide us a Fourier analysis for suitable functions defined on $\Z_p$ in such a way that the formula $$\varphi(u) = \sum_{\tau \in \widehat{\Z}_p^d} \widehat{\varphi}(\tau) \chi_p (\tau  u),$$holds almost everywhere in $\Z_p$. Here $\mathcal{F}_{\Z_p^d}[\varphi]$ denotes the Fourier transform of $f$, in turn defined as $$\mathcal{F}_{\Z_p^d}[\varphi](\tau):= \int_{\Z_p^d} \varphi(u) \overline{\chi_p (\tau  u)}du,$$where $du$ is the normalised Haar measure on $\Z_p^d$. 

\subsection{The Heisenberg group over $\Z_p$}
Let $p>2$ be a prime number. Let us denote by $\mathbb{H}_d (\Z_p)$ the $(2d+1)$-dimensional Heisenberg group over $\Z_p$, or simply $\mathbb{H}_d $ for short, here defined as \[
\mathbb{H}_{d}(\Z_p)= \left\{
  \begin{bmatrix}
    1 & x^t & z \\
    0 & I_{d} & y \\
    0 & 0 & 1 
  \end{bmatrix}\in \mathrm{GL}_{d+2}(\Z_p) \, : \, x , y \in \Z_p^d, \, z \in \Z_p \right\}. 
\]
Clearly $\mathbb{H}_{d}(\Z_p)$ is a compact analytic $d$-dimensional manifold, which is homeomorphic to $\Z_p^{2d+1}$. Moreover, the operations on $\mathbb{H}_{d}(\Z_p)$ are analytic functions, making $\mathbb{H}_{d}(\Z_p)$ a $p$-adic Lie group. Let us denote by $\mathfrak{h}_{d}(\Z_p)$ its associated $\Z_p$-Lie algebra. We can write explicitly \[
\mathfrak{h}_{d}(\Z_p)= \left\{
  \begin{bmatrix}
    0 & \textbf{a}^t & c - \frac{\textbf{a} \cdot \textbf{b}}{2} \\
    0 & 0_{d} & \textbf{b} \\
    0 & 0 & 0 
  \end{bmatrix}\in \mathcal{M}_{d+2}(\Z_p) \, : \, \textbf{a} , \textbf{b} \in \Z_p^d, c \in \Z_p \right\}. 
\]Recall how for an element of the Lie algebra \[u:= \begin{bmatrix}
    0 & \textbf{a}^t & c - \frac{\textbf{a} \cdot \textbf{b}}{2}\\
    0 & 0_{d} & \textbf{b} \\
    0 & 0 & 0 
  \end{bmatrix},\]the exponential map evaluates to \[\textbf{exp} (u) = \begin{bmatrix}
    1 & \textbf{a}^t & c  \\
    0 & I_{d} & \textbf{b} \\
    0 & 0 & 1 
  \end{bmatrix}.\]The exponential map transform sub-ideals of the Lie algebra $\mathfrak{h}_{d} $ to subgroups of $\mathbb{H}_{d} $. Actually, we can turn the exponential map into a group homomorphism by using the Baker–Campbell–Hausdorff formula. Let us define the operation ``$\star$" on $\mathfrak{h}_{d}$ by $$X \star Y:= X + Y + \frac{1}{2} [X,Y].$$ Then clearly $(\mathfrak{h}_{d}  , \star) \cong \mathbb{H}_d $ is a profinite topological group, and it can be endowed with the sequence of subgroups $J_n := (  \mathfrak{h}_{d}(p^n\Z_p),\star)$, where $$ \mathfrak{h}_{d}(p^n\Z_p)= p^n \Z_p X_1 +...+ p^n \Z_p X_d + p^n \Z_p Y_1 + ...+ p^n\Z_p Y_{d} + p^{n} \Z_p Z,$$so $\mathbb{H}_d$ is a compact Vilenkin group, together with the sequence of compact open subgroups $$G_n := \mathbb{H}_d (p^n \Z_p)= \textbf{exp}(\mathfrak{h}_{d}(p^n\Z_p)), \,\,\, n \in \N_0.$$
  Notice how the sequence $\mathscr{G}=\{G_n\}_{n \in \N_0}$ forms a basis of neighbourhoods at the identity, so the group is metrizable, and we can endow it with the natural ultrametric \[ |(\mathbf{x},\mathbf{y},z)\star(\mathbf{x}',\mathbf{y}',z')^{-1}|_{\mathscr{G}} :=\begin{cases} 0 & \, \, \text{if} \, x=y, \\ |G_n| = p^{-n(2d+1)}  & \, \, \text{if} \, (\mathbf{x},\mathbf{y},z) \star(\mathbf{x}',\mathbf{y}',z')^{-1} \in G_n \setminus G_{n+1}.\end{cases}\]   
Nevertheless, instead of this ultrametric we will use the $p$-adic norm $$\| (\mathbf{x},\mathbf{y},z) \|_p:= \max \{\|\mathbf{x}\|_p, \| \mathbf{y} \|_p , |z|_p  \}.$$ Notice that how $\|(\mathbf{x},\mathbf{y},z) \|_p^{2d+1} = |(\mathbf{x},\mathbf{y},z)|_{\mathscr{G}},$ for any $(\mathbf{x},\mathbf{y},z) \in \mathbb{H}_d.$
\subsection{Directional VT operators}
One important idea from the theory of differential and pseudo-differential operators on Lie groups, is the correspondence between directional derivatives and elements of the Lie algebra. However, in the $p$-adic case, there are plenty of non-trivial locally constant functions, due to the fact that $p$-adic numbers are totally disconnected. This means that the usual notion of derivative does not apply, and therefore we need to find an alternative kind of operators to talk about differentiability on these groups. A first approach to this problem can be the \emph{Vladimitov-Taibleson operator} \cite{Dragovich2023, Dragovich2017}, which we define for general compact $\K$-Lie groups as follows:    

\begin{defi}\label{defiVToperator}\normalfont
Let $\K$ be a non-archimedean local field with ring of integers $\mathscr{O}_\K$, prime ideal $\mathfrak{p}=\textbf{p} \mathscr{O}_\K$ and residue field $\mathbb{F}_q \cong \mathscr{O}_\K/\textbf{p} \mathscr{O}_\K$. Let  $\mathbb{G} \leq \mathrm{GL}_m (\mathscr{O}_\K)$ be a compact $d$-dimensional $\K$-Lie group. We define the \emph{Vladimirov–Taibleson operator} on $\mathbb{G}$  via the formula \[
D^\alpha f(\mathbf{x}) := \frac{1 - q^\alpha}{1 - q^{- (\alpha + d)}} \int_{\mathbb{G}} \frac{f (\mathbf{x} \star \mathbf{y}^{-1}) - f(\mathbf{x})}{\|\mathbf{y} \|_\K^{ \alpha + d}} d\mathbf{y},
\]where \[\| \mathbf{y} \|_\K := \begin{cases}
    1, \, & \, \, \text{if} \, \, \mathbf{y} \in \mathbb{G} \setminus \mathrm{GL}_m (\textbf{p}\mathscr{O}_\K), \\ q^{-n}, \, & \, \, \text{if} \, \, \mathbf{y} \in \mathrm{GL}_m (\textbf{p}^n\mathscr{O}_\K) \setminus \mathrm{GL}_m (\textbf{p}^{n+1}\mathscr{O}_\K). 
\end{cases}\]Here $dy$ denotes the unique normalized Haar measure on $\mathbb{G}$. Sometimes it will be convenient to consider the operator \[
\mathbb{D}^\alpha f(\mathbf{x}) :=\frac{1-q^{-d}}{1-q^{-(\alpha +d)}}f(\mathbf{x}) + \frac{1 - q^\alpha}{1 - q^{- (\alpha + d)}} \int_{\mathbb{G}} \frac{f (\mathbf{x} \star \mathbf{y}^{-1}) - f(\mathbf{x})}{\|\mathbf{y} \|_\K^{ \alpha + d}} d\mathbf{y},
\]
\end{defi}

The Vladimirov–Taibleson operator can be considered as a fractional Laplacian for functions on totally disconnected spaces, and it provides a first notion of differentiability. However, for functions of several variables it is natural to consider the differentiability of the function in each variable, or in a certain given direction. For that reason, we introduce the following definition.

 \begin{defi}\label{defcompactdirectionalVT}\normalfont
      Let $\mathfrak{g}$ be the $\mathscr{O}_\K$-Lie module associated to $\mathbb{G}$, and assume that $\mathfrak{g}$ is nilpotent. Let $\alpha>0$. Given a $V \in \mathfrak{g}$, we define the \emph{directional VT operator of order $\alpha$ in the direction of} $V$ as the linear invariant operator $\partial_V^\alpha$ acting on smooth functions via the formula $$\partial_V^\alpha f (\mathbf{x}) := \frac{1 - q^{\alpha}}{1-q^{-(\alpha+1)}} \int_{\mathscr{O}_\K} \frac{f(\mathbf{x} \star \textbf{exp}(tV)^{-1}) - f(x)}{|t|_\K^{ \alpha + 1}}dt, \, \, \, f \in \mathcal{D}(G).$$ 
 \end{defi} 
  \begin{rem}
 Directional VT operators are interesting because they associate a certain pseudo-differential operator to each element of the Lie algebra. Nonetheless, it is important to remark how this association does not follow the same patter as in the locally connected case, where the correspondence between vectors and operators preserves the Lie algebra structure.
 \end{rem}
In order to study the behaviour of the directional VT operators, let us introduce an important definition:

\begin{defi}\label{defihypo}\normalfont
Let $\mathbb{G}$ be a compact $p$-adic Lie group, and let $\mathrm{Rep}(\mathbb{G})$ be the collection of all unitary continuous finite-dimensional representations of $\mathbb{G}$.
\begin{itemize}
    \item A \emph{symbol $\sigma$} is a mapping $$\sigma: \mathbb{G} \times \mathrm{Rep}(\mathbb{G}) \to \bigcup_{[\pi] \in \mathrm{Rep}(\mathbb{G})} \mathcal{L}(\mathcal{H}_{\pi}), \,\,\,\, (\mathbf{x},[\pi]) \mapsto \sigma(\mathbf{x}, \pi) \in \mathcal{L}(\mathcal{H}_\pi).$$Given a symbol on $\mathbb{G}$, we define its associated pseudo-differential operator as the linear operator $T_\sigma$ acting on $\mathcal{D}(\mathbb{G})$ via the formula $$T_\sigma f(\mathbf{x}):= \sum_{[\xi] \in \widehat{G}} d_\xi Tr[ \xi(\mathbf{x}) \sigma (\mathbf{x},\xi) \widehat{f} (\xi)].$$
    \item Conversely, given a densely defined linear operator $T : \mathcal{D} (\mathbb{G}) \subset D(T) \to \mathcal{D}'(\mathbb{G})$, we define its associated symbol via the formula $$\sigma(\mathbf{x}, [\pi]) = \pi^* (\mathbf{x}) T \pi (\mathbf{x}).$$ 
    \item Let $T_\sigma : \mathcal{D} (\mathbb{G}) \subset D(T) \to \mathcal{D}'(\mathbb{G})$ be a densely defined linear operator. We say that $T_\sigma$ is globally hypoelliptic if the condition $T_\sigma f = g$ with $f \in \mathcal{D}'(\mathbb{G})$ and $g \in \mathcal{S} (\mathbb{G})$ implies that $f \in \mathcal{S} (\mathbb{G})$. Here the \emph{Schwartz} space is defined as the collection of $L^2$-functions such that $$\| \widehat{f}(\xi) \|_{HS} \lesssim  \langle \xi \rangle_{\mathbb{G}}^{-k},\,\,\, \text{for all} \, k \in \N_0,$$where $\widehat{f}(\xi) := \int_{\mathbb{G}} f(\mathbf{x}) \xi^* (\mathbf{x}) d\mathbf{x},$ and $\langle \xi \rangle_{\mathbb{G}}$ denotes the eigenvalue of the Vladimirov-Taibleson operator $\mathbb{D}^1$ defined in Definition \ref{defiVToperator}, associated to the class $[\xi] \in \widehat{\mathbb{G}}$. 
\end{itemize}
\end{defi}Just to give an example, in the particular case when $G=\mathfrak{g}=\Z_p^d$, these operators take the form $$\partial_V^\alpha f (x) := \frac{1 - p^{\alpha}}{1-p^{-(\alpha+1)}}\int_{\Z_p^d} \frac{f(x -tV) - f(x)}{| t |_p^{ \alpha + 1}}dt, $$and one can easily compute its associated symbol: 
\[\sigma_{\partial_{V}^\alpha}(\xi) = \begin{cases}
    0, \, & \, \, \text{if} \, \, |V \cdot \xi|_p \leq 1,\\|V \cdot \xi|_p^\alpha - \frac{1 - p^{-1}}{1 - p^{- (\alpha + 1)}}  & \, \, \text{if} \, \, |V \cdot\xi|_p>1.
\end{cases}
 \]If we define $\partial_{x_i}^\alpha := \partial_{e_i}^\alpha $, where $e_i$, $1 \leq i \leq d$, are the canonical vectors of $\Q_p^d$, then \[\sigma_{\partial_{x_i}^\alpha}(\xi) = \begin{cases}
    0, \, & \, \, \text{if} \, \, |\xi_i|_p=1,\\| \xi_i|_p^\alpha - \frac{1 - p^{-1}}{1 - p^{- (\alpha + 1)}}  & \, \, \text{if} \, \, |\xi_i|_p>1,
\end{cases}
 \]which resembles the symbol of the usual partial derivatives on $\R^d$, justifying that way our choice of notation. However, we want to be emphatic about the fact that these directional VT operators do not preserve the Lie algebra structure, and they are not derivatives, but rather some special kind of pseudo-differential operators which we will study with the help of the Fourier analysis on compact groups.

Before proceeding to the next section, let us introduce some notation. 

\begin{defi}\label{definotation}\normalfont
\,
\begin{itemize}
        \item The symbol $\mathrm{Rep}(\mathbb{H}_d)$ will denote the collection of all unitary finite-dimensional representations of $\mathbb{H}_d$. We will denote by $\widehat{\mathbb{H}}_d$ the \emph{unitary dual of $\mathbb{H}_d$}.
        \item Let $K$ be a normal sub-group of $\mathbb{H}_d$. We denote by $K^\bot$ the \emph{anihilator of $K$}, here defined as $$K^\bot:= \{ [\pi] \in \mathrm{Rep}(\mathbb{H}_d) \, : \, \pi|_{K}=I_{d_\pi} \}.$$Also, we will use the notation $$B_{\widehat{\mathbb{H}}_d}(n):=\widehat{\mathbb{H}}_d \cap \mathbb{H}_d(p^n\Z_p)^\bot,$$and $\widehat{\mathbb{H}}_d(n):= B_{\widehat{\mathbb{H}}_d}(n) \setminus B_{\widehat{\mathbb{H}}_d}(n-1).$ 
        \item We say that a function $f:\mathbb{H}_d \to \C$ is a \emph{smooth function}, if $f$ is a locally constant function with a fixed index of local constancy, i.e., there is an $n_f \in \N_0$, which we always choose to be the minimum possible, such that $$f((\mathbf{x},\mathbf{x},z)\star(\mathbf{x}',\mathbf{y}',z')) = f(\mathbf{x},\mathbf{y},z), \, \, \,\text{ for all} \,\,\, (\mathbf{x}',\mathbf{y}',z) \in \mathbb{H}_d(p^{n_f} \Z_p).$$
        Here $\star$ is the operation on $\mathbb{H}_d$ which we simply take as $$(\mathbf{x},\mathbf{y},z)\star(\mathbf{x}',\mathbf{y}',z):=(\mathbf{x} + \mathbf{x}', \mathbf{y} + \mathbf{y}', z + z' + \mathbf{x} \cdot \mathbf{y}').$$ We will denote by $\mathcal{D}(\mathbb{H}_d)$ the collection of all smooth functions on $\mathbb{H}_d$, and $\mathcal{D}_n (\mathbb{H}_d)$ will denote the collection of smooth functions with index of local constancy equal to $n \in \N_0$. 
    \end{itemize}
\end{defi}
\section{Representation theory of the Heisenberg group}

\subsection{The idea behind: n=0, 1, 2}$\mathbb{H}_d$ is probably the simplest example of a (non-commutative) compact Vilenkin group. Just like for any other compact Vilenkin group, the representations of $\mathbb{H}_d$ have a non-trivial kernel, which is a compact open subgroup, and it must contain some of the subgroups $G_n =\mathbb{H}_d(p^n \Z_p)$. This means that each unitary irreducible representation of $G=\mathbb{H}_d$, whose associated matrix coefficients are therefore smooth functions, must descend to a representation of one of the groups $G/G_n \cong \mathbb{H}_d(\mathbb{F}_{p^n})$, $n \in \N_0$. We can use this information and some intuition to figure out the dimensions of the unitary irreducible representations.  We will do it here by finding all the elements in $B_{\widehat{\mathbb{H}}_d}(n)$ for each $n \in \N_0$.

\textbf{(i)}. $\widehat{\mathbb{H}}_d (0)$, the collection of unitary irreducible representations that are trivial on $G_0 = \mathbb{H}_d (\Z_p)$, contains only the identity representation.

\textbf{(ii)}.$B_{\widehat{\mathbb{H}}_d} (1)$, the collection of unitary irreducible representations that are trivial on $G_1 = \mathbb{H}_d (p \Z_p)$, are precisely those representations which descend to an element of the unitary dual of $G / G_1 \cong \mathbb{H}_d (\Z_p / p \Z_p) \cong \mathbb{H}_d (\mathbb{F}_p) .$ In general, any representation of $\mathbb{H}_d$ is the representation $Ind_\lambda$ induced by some central character $e^{2 \pi i \{ \lambda z \}_p}$, $\lambda \in \Q_p / \Z_p$. If the character is trivial on $(G_0 \setminus G_1) \cap \mathcal{Z}(\mathbb{H}_d)$, i.e. when $|\lambda|_p = 1$, then the representations induced by the trivial central character are all one-dimensional, and they must have the form $$\chi_{\xi, \eta} (\mathbf{x}, \mathbf{y} , z) := e^{2 \pi i \{ \mathbf{x} \cdot \xi + \mathbf{yt} \cdot \eta \}_p}, \, \, \, (\xi , \eta) \in \widehat{\Z}_p^d \times \widehat{\Z}_p^d.$$In particular, the condition $\chi_{\xi, \eta} \in B_{\widehat{\mathbb{H}}_d} (1) $ implies that $\| (\xi , \eta)\|_p \leq p$. In the case when the character is not trivial on $(G_0 \setminus G_1) \cap \mathcal{Z}(\mathbb{H}_d)$,  and we also assume that $Ind_\lambda \in \widehat{\mathbb{H}}_d (1)$, the only possibility is that $| \lambda |_p =p$, and $Ind_\lambda$ descend to a non-commutative representation of $\mathbb{H}_d (\mathbb{F}_p).$ These are representations of dimension $p^d$ which can be realized in the following sub-space of $L^2 (\mathbb{H}_d)$: 
$$\Tilde{\mathcal{H}}_\lambda := \spn_\C \{ \varphi_h (a, b,c ) :=  p^{d/2} e^{2 \pi i \{ \lambda c \}_p} \1_{h + p \Z_p^d} (a) \, : \, h \in \Z_p^d / p \Z_p^d  \},$$via the formula $$\tilde{\pi}_\lambda (\mathbf{x}, \mathbf{y} , z) \varphi (a,b,c):= \varphi((a,b,c)(\mathbf{x}, \mathbf{y} , z)) = e^{2 \pi i \{ \lambda(z + a \cdot \mathbf{y}) \}_p} \varphi (a + \mathbf{x} ,b,c) .$$For simplicity, we will use instead the realization $$\pi_\lambda (\mathbf{x}, \mathbf{y} , z) \varphi (u):= e^{2 \pi i \{ \lambda(z + u \cdot \mathbf{y}) \}_p} \varphi (u + \mathbf{x} ) , \, \, \, \varphi \in \mathcal{H}_\lambda \subset L^2 (\Z_p^d),$$where
$$\mathcal{H}_\lambda := \spn_\C \{ \varphi_h (u ) :=  p^{d/2}  \1_{h + p \Z_p^d} (u) \, : \, h \in \Z_p^d / p \Z_p^d  \},$$and it is easy to check how this gives an unitary representation of $\mathbb{H}_d (\Z_p)$. To check that it is irreducible, we calculate the matrix coefficients in order to have a simple expression for the trace: 
\begin{align*}
    (\pi_{ \lambda})_{h h'}(\mathbf{x}, \mathbf{y} , z)&= (\pi_{\lambda} \varphi_h , \varphi_{h'})_{L^2 (\Z_p^d)} \\ &= p^{ d } \int_{\Z_p^d} e^{2 \pi i \{ \lambda(z + u \cdot \mathbf{y}) \}_p} \1_{h+p \Z_p^d} (u +\mathbf{x}) \1_{h'+p \Z_p^d} (u) du \\ &=p^{ d }  \1_{h - h' +p \Z_p^d} (\mathbf{x}) \int_{h' + p \Z_p^d} e^{2 \pi i \{ \lambda(z + u \cdot \mathbf{y})  \}_p}  du \\ &= e^{ 2 \pi i \{ \lambda(z +  h' \cdot \mathbf{y} ) \}_p} \1_{h - h' +p \Z_p^d }(\mathbf{x}) \Big( p^{d} \int_{p \Z_p^d} e^{2 \pi i \{ \lambda  u \cdot \mathbf{y}\}_p } du \Big) \\ &=e^{ 2 \pi i \{ \lambda(z +  h' \cdot \mathbf{y} )  \}_p} \1_{h - h'+p \Z_p^d}(\mathbf{x}) .
\end{align*}
In this way, the character $\chi_{\pi_\lambda}$ is given by the expression      
$$\chi_{\pi_\lambda}(\mathbf{x}, \mathbf{y} , z) = \sum_{h \in  \Z_p^d/p\Z_p^d} e^{ 2 \pi i \{ \lambda(z +  h  \cdot \mathbf{y}  ) \}_p} \1_{p \Z_p^d}(\mathbf{x}) = p^d e^{ 2 \pi i \{ \lambda z  \}_p} \1_{p \Z_p^d}(\mathbf{x}) \1_{p \Z_p^d}(\mathbf{y}), $$thus the irreductibility of $[\pi_\lambda]$ is proven by the condition  $$\int_{\mathbb{H}_d}|\chi_{\pi_\lambda}(\mathbf{x}, \mathbf{y} , z)|^2 d\mathbf{x} \, d\mathbf{y} \, dz = p^{2d} \int_{\mathbb{H}_d}|\1_{p \Z_p^d}(\mathbf{x}) \1_{p \Z_p^d}(\mathbf{y})|^2 d\mathbf{x} \, d\mathbf{y} \, dz = 1.$$
Notice how these are all the elements of $B_{\widehat{\mathbb{H}}_d}(1)=\widehat{\mathbb{H}}_d \cap \mathbb{H}_d(p^1 \Z_p)^\bot$ since 
     $$\sum_{\| (\xi ,\eta) \|_p \leq p} \dim(\chi_{\xi , \eta})^2 + \sum_{| \lambda |_p = p } \dim(\pi_\lambda)^2 = \sum_{\| (\xi ,\eta) \|_p \leq p} 1^2 + \sum_{| \lambda |_p = p } (p^d)^2 =p^{2d+1}= |G/G_1|.$$

\textbf{(iii)}. Now we want to describe the elements of $B_{\widehat{\mathbb{H}}_d}(2)=\widehat{\mathbb{H}}_d \cap \mathbb{H}_d (p^2\Z_p)^\bot$. We start again with the representations $\chi_{\xi , \eta}$ trivial on the center, which in order to be in $B_{\widehat{\mathbb{H}}_d}(2)$ need to fulfill $\| (\xi ,\eta) \|_p \leq p^2.$ So we have $(p^2)^{2d}$ one-dimensional representations.  For the representations $ Ind_\lambda$ induced by a non-trivial central character we have the following possibilities.
\begin{itemize}
    \item $Ind_\lambda = \pi_{(\xi , \eta, \lambda)} =  \chi_{\xi , \eta} \otimes \pi_\lambda$, where $|\lambda|_p = p$ and $\| (\xi, \eta) \|_p \leq p^2$. If $\| (\xi ,\eta) \|_p = 1$ then $\pi_{(\xi , \eta, \lambda)} \in \widehat{\mathbb{H}}_d (1)$, and we already considered these representations. If $\|(\xi, \eta ) \|_p>1$, we can realize explicitly the unitary representation $\chi_{\xi , \eta} \otimes \pi_\lambda$ via the formula $$\pi_{(\xi, \eta, \lambda)} (\mathbf{x}, \mathbf{y} , z) \varphi (u):= e^{2 \pi i \{\xi \cdot \mathbf{x}+ \eta \cdot \mathbf{y} + \lambda(z + u \cdot \mathbf{y})  \}_p} \varphi (u + \mathbf{x} ) , \, \, \, \varphi \in \mathcal{H}_\lambda \subset L^2 (\Z_p^d),$$where
$$\mathcal{H}_\lambda := \spn_\C \{ \varphi_h (u ) :=  p^{d/2}  \1_{h + p \Z_p^d} (u) \, : \, h \in \Z_p^d / p \Z_p^d  \}.$$ In this case $d_{(\xi , \eta , \lambda)} := \dim_\C (\pi_{(\xi , \eta, \lambda)})=p^d$, and the associated characters are \begin{align*}
    \chi_{\pi_{(\xi, \eta, \lambda)}}(\mathbf{x}, \mathbf{y} , z) &= \sum_{h \in  \Z_p^d/p\Z_p^d} e^{ 2 \pi i \{ \lambda(z +  h \cdot \mathbf{y} ) + ( \xi \cdot \mathbf{x} +  \eta \cdot \mathbf{y}) \}_p} \1_{p \Z_p^d}(\mathbf{x}) \\ &= p^d e^{ 2 \pi i \{ \lambda z + (x \xi + y \eta) \}_p} \1_{p \Z_p^d}(x) \1_{p \Z_p^d}(y),
\end{align*}
so that $\| \chi_{\pi_{(\xi, \eta, \lambda)}} \|_{L^2 (\mathbb{H}_d)}=1$, proving that each $\pi_{(\xi, \eta, \lambda)}$ is irreducible. Also, notice how $$p^d e^{ 2 \pi i \{ \lambda z + (x \xi + y \eta) \}_p} \1_{p \Z_p^d}(x) \1_{p \Z_p^d}(y) = p^d e^{ 2 \pi i \{ \lambda z + (x \xi_2 + y \eta_2 ) \}_p} \1_{p \Z_p^d}(x) \1_{p \Z_p^d}(y),$$so that there are exactly $p^{2d}(p-1)$ different equivalent classes of representations among the tensor products $\chi_{\xi , \eta} \otimes \pi_\lambda$, corresponding to $\| (\xi, \eta) \|_p = p^2$, with $(\xi , \eta) \in  \Q_p^{2d} / p^{\vartheta(\lambda)}\Z_p^{2d}$, and $\| (\xi, \eta) \|_p =1$.  
     \item For the remaining $p(p-1)$ elements $\lambda$ with $|\lambda|_p = p^2$, $Ind_\lambda \in \widehat{\mathbb{H}}_d (2)$, and their associated induced representation $Ind_\lambda$ have to descend to one of the non-commutative representations in the unitary dual of $$G/G_2 = \mathbb{H}_d (\Z_p) / \mathbb{H}_d (p^2 \Z_p) \cong \mathbb{H}_d (\mathbb{F}_{p^2}),$$which are all representations of dimension $(p^2)^d$. These can be realized explicitly once again via the formula $$\pi_{(1, 1, \lambda)} (x, y , z) \varphi (u):= e^{2 \pi i \{ \lambda(z + u y)  \}_p} \varphi (u + x ) , \, \, \, \varphi \in \mathcal{H}_\lambda \subset L^2 (\Z_p^d),$$where this time
$$\mathcal{H}_\lambda := \spn_\C \{ \varphi_h (u ) :=  p^{d}  \1_{h + p^2 \Z_p^d} (u) \, : \, h \in \Z_p^d / p^2 \Z_p^d  \}.$$    
\end{itemize}

This concludes the description of $B_{\widehat{\mathbb{H}}_d}(2)$ since \begin{align*}
    \sum_{\|(\xi,\eta, \lambda)\|_p \leq p^2} d_{(\xi , \eta , \lambda)}^2 &= \sum_{\|(\xi,\eta)\|_p \leq p^2, \, |\lambda|_p=1} d_{(\xi , \eta , \lambda)}^2 + \sum_{\|(\xi,\eta)\|_p > p, \, \, |\lambda|_p=p} d_{(\xi , \eta , \lambda)}^2 +\sum_{ |\lambda|_p=p^2,  \, \,Ind_\lambda \in \widehat{\mathbb{H}}_d (\mathbb{F}_{p^2}) } d_{(\xi , \eta , \lambda)}^2 \\ &= \sum_{\|(\xi,\eta)\|_p \leq p^2, \, |\lambda|_p=1} 1^2 + \sum_{\|(\xi,\eta)\|_p \neq p, \, \, |\lambda|_p=p} (p^d)^2  \sum_{ \|(\xi , \eta)\|_p=1, \, |\lambda|_p=p^2 } ((p^2)^d)^2 \\ &= (p^2)^{2d} + p^{2d}(p-1)(p^d)^2 +  (p-1)p((p^2)^d)^2 \\&= (p^2)^{2d+1} = |G/G_2|.
\end{align*} 
Before advancing to the next sub-section, let us recall an important convention:
\begin{rem}
    Remember how in this paper we are identifying each equivalence class $\lambda$ in $\widehat{\Z}_p \cong \Q_p / \Z_p$, with its associated representative in the complete system of representatives $$\{1\} \cup \big\{ \sum_{k =1}^\infty \lambda_k p^{-k} \, : \, \, \text{only finitely many $\lambda_k$ are non-zero} \big\}.$$
\end{rem}
\subsection{The general case}
Let us summarize the process we just employed to obtain the representations in $\widehat{\mathbb{H}}_d \cap \mathbb{H}_d (p^2\Z_p)^\bot$.

\begin{enumerate}
    \item First we obtained the elements of $B_{\widehat{\mathbb{H}}_d} (1)$, which are simply the unitary irreducible representations of $\mathbb{H}_d(\Z_p)/\mathbb{H}_d(p \Z_p) \cong \mathbb{H}_d(\mathbb{F}_p)$. By doing this we see how there are two kind of representations of $\mathbb{H}_d$, characters $\chi_{\xi , \eta}$ and non-commutative representations $\pi_\lambda$ induced by a central character. We can consider their tensor products $\chi_{\xi , \eta} \otimes \pi_\lambda$ too, but we will obtain representations equivalent to the non-commutative representations $\pi_\lambda$, because they will share the same character.    
    \item Second, since the representations need to be trivial on $\mathbb{H}_d (p^2\Z_p)$, we must have $\|(\xi , \eta , \lambda) \|_p \leq p^2$. We got some one–dimensional representations, induced by the trivial central character, which are characters $\chi_{\xi , \eta}$ of $\mathbb{H}_d / \mathcal{Z} (\mathbb{H}_d) \cong \Z_p^{2d}$, such that $\|(\xi , \eta) \|_p \leq p^2$. There are exactly $(p^2)^{2d}$ of them.  
    \item Third, we considered tensor products. The representations $\chi_{\xi, \eta} \otimes \pi_\lambda $, $|\lambda|_p=p$, are equivalent to $\pi_\lambda$ if $\| (\xi , \eta)\|_p \leq p$, and they define non-equivalent unitary irreducible representations when $\|(\xi , \eta)\|_p = p^2$. So, in total we have $p^{2d}(p-1)$ representations of the form $\chi_{\xi, \eta} \otimes \pi_\lambda$, and for them we have $d_{(\xi, \eta, \lambda)}=p^d$. 
    \item If $|\lambda|_p=p^2$, there are exactly $p(p-1)$ of these $\lambda$, and their induced representations $Ind_\lambda$ should descend to a unitary irreducible non-commutative representation of $$\mathbb{H}_d (\Z_p) / \mathbb{H}_d (p^2 \Z_p) \cong \mathbb{H}_d (\mathbb{F}_{p^2}),$$which are all of dimension $(p^2)^d$.  
\end{enumerate}
Summing up, we can write: $$B_{\widehat{\mathbb{H}}_d}(2) = \{ \chi_{\xi , \eta} \otimes \pi_\lambda \, :\, \| (\xi , \eta , \lambda) \|_p \leq p^2, \,\, \text{and} \,  \|(\xi, \eta)\|=1, \, \, \text{or} \, \, \,  (\xi, \eta) \in \Q_p^{2d}/ p^{\vartheta(\lambda)}\Z_p\}.$$  
We can use the arguments collected so far to obtain all the unitary irreducible representations of $\mathbb{H}_d$. To illustrate the general process, let us start with the elements of $B_{\widehat{\mathbb{H}}_d}(3)=\widehat{\mathbb{H}}_d \cap \mathbb{H}_d (p^3\Z_p)^\bot$. 
\begin{itemize}
    \item We have $(p^{3})^{2d}$ characters $\chi_{\xi , \eta}$ corresponding to the $(\xi , \eta) \in \widehat{\Z}_p^{2d}$ such that $\| (\xi , \eta) \|_p \leq p^3$. If we identify the trivial central character with $\lambda = 1$, then these characters are indexed by the set $$ \{ (\xi , \eta , 1) \in \widehat{\Z}_p^{2d+1} \, :\, \| (\xi , \eta , 1) \|_p \leq p^3, \,\, \text{and} \, \, (\xi, \eta) \in \Q_p^{2d} / \Z_p^{2d} \}.$$ 
    \item More generally, for any $\lambda \in \widehat{\Z}_p$, $1<|\lambda|_p \leq p^3$, we can check how the representation $$\pi_{\lambda} (x, y , z) \varphi (u):= e^{2 \pi i \{ \lambda(z + u y)  \}_p} \varphi (u + x ) , \, \, \, \varphi \in \mathcal{H}_\lambda \subset L^2 (\Z_p^d),$$where$$\mathcal{H}_\lambda := \spn_\C \{ \varphi_h (u ) :=  |\lambda|_p^{d/2}  \1_{h + |\lambda|_p \Z_p^d} (u) \, : \, h \in \Z_p^d / |\lambda|_p \Z_p^d  \},$$is unitary and irreducible, since its associated character is given by 

    $$\chi_{\pi_\lambda} (x,y,z) = |\lambda|_p^d e^{2 \pi i \{ \lambda z\}_p}\1_{p^{-\vartheta(\lambda)} \Z_p^d} (x)\1_{p^{-\vartheta(\lambda)} \Z_p^d} (y).$$The irreducibility condition still holds when we take tensor products with the commutative characters $\chi_{\xi, \eta}$, as long as $(\xi , \eta) \in \Q_p^{2d} / p^{\vartheta(\lambda)} \Z_p^{2d}$. Also, for a fixed $1 < |\lambda|_p \leq p^3$  we get $$|\{ (\xi , \eta) \in \widehat{\Z}_p^{2d} \, :\, \|(\xi, \eta) \|_p \leq p^3, \, (\xi , \eta) \in \Q_p^{2d} / p^{\vartheta(\lambda)} \Z_p^{2d} \}| = (p^3)^{2d}|\lambda|_p^{-2d}.$$
     In this way \begin{align*}
         \sum_{(\xi , \eta, \lambda) \in B_{\widehat{\mathbb{H}}_d}(3)} d_{(\xi , \eta , \lambda)}^2 &= \sum_{1 \leq |\lambda|_p \leq p^3 } \sum_{\{\| (\xi , \eta) \|_p \leq p^3 \, : (\xi , \eta) \in \Q_p^{2d} / p^{\vartheta(\lambda)} \Z_p^{2d}\}} d_{(\xi , \eta , \lambda)}^2 \\ 
         &= \sum_{1 \leq |\lambda|_p \leq p^3 } \sum_{\{ \| (\xi , \eta) \|_p \leq p^3 \, : \, (\xi , \eta) \in \Q_p^{2d} / p^{\vartheta(\lambda)} \Z_p^{2d}\}} (|\lambda|_p^d)^2 \\ 
         &= \sum_{1 \leq |\lambda|_p \leq p^3 } |\lambda|_p^{2d} |\{ \|(\xi , \eta)\| \leq p^3 \, : \,(\xi , \eta) \in \Q_p^{2d} / p^{\vartheta(\lambda)} \Z_p^{2d} \}| \\ &= (p^3)^{2d}+ \sum_{1 < |\lambda|_p \leq p^3 } |\lambda|_p^{2d}  (p^3)^{2d}|\lambda|_p^{-2d} = (p^{3})^{2d+1}.
     \end{align*}
\end{itemize}
More generally, we can identify $\widehat{\mathbb{H}}_d$ with the set $$\widehat{\mathbb{H}}_d = \{ (\xi , \eta , \lambda) \in \widehat{\Z}_p^{2d+1} \, : \, (\xi , \eta) \in \Q_p^{2d} / p^{\vartheta(\lambda)} \Z_p^{2d} \},$$so that the annihilators $B_{\widehat{\mathbb{H}}_d}(n)=\widehat{\mathbb{H}}_d \cap \mathbb{H}_d (p^n\Z_p)^\bot$ coincide with the balls 
$$B_{\widehat{\mathbb{H}}_d}(n) = \{ (\xi , \eta , \lambda) \in \widehat{\mathbb{H}}_d \, : \, \| (\xi , \eta , \lambda) \|_p \leq p^n \},$$also
$$\widehat{\mathbb{H}}_d (n)= \{ (\xi , \eta , \lambda) \in \widehat{\mathbb{H}}_d \, : \, \| (\xi , \eta , \lambda) \|_p = p^n \} .$$
We can easily check how these are indeed all the desired representations, since \begin{align*}
         \sum_{(\xi , \eta, \lambda) \in B_{\widehat{\mathbb{H}}_d}(n)} d_{(\xi , \eta , \lambda)}^2 &= \sum_{1 \leq |\lambda|_p \leq p^n } \sum_{\{ \| (\xi , \eta) \|_p \leq p^n \, : \, (\xi , \eta) \in \Q_p^{2d} / p^{\vartheta(\lambda)} \Z_p^{2d}\}} (|\lambda|_p^d)^2 \\ 
         &= (p^n)^{2d} + \sum_{1 < |\lambda|_p \leq p^n } \sum_{\{ \| (\xi , \eta) \|_p \leq p^n \, : \,  \| (\xi , \eta) \|_p \leq p^n \, : \, (\xi , \eta) \in \Q_p^{2d} / p^{\vartheta(\lambda)} \Z_p^{2d}\}} (|\lambda_p|^d)^2 \\ 
         &= (p^n)^{2d} + \sum_{1< |\lambda|_p \leq p^n } (|\lambda|_p^d)^2|\{ \|(\xi , \eta)\| \leq p^n \, : (\xi , \eta) \in \Q_p^{2d} / p^{\vartheta(\lambda)} \Z_p^{2d} \}| \\ &=(p^n)^{2d} + \sum_{1 < |\lambda|_p \leq p^n } |\lambda|_p^{2d} ((p^{n})^{2d}|\lambda|_p^{-2d}) \\ &=(p^n)^{2d} + (p^n)^{2d}(p^n - 1) = (p^{n})^{2d+1} = |G/G_n|=| \mathbb{H}_d (\mathbb{F}_{p^n})|.
     \end{align*}
This gives a complete description of the unitary dual of $\mathbb{H}_d$. 

\subsection{Matrix coefficients and Fourier series.}
From the analysis in the previous subsection, and the Fourier analysis in general compact groups, we obtain the following Fourier series representation for functions on $\mathbb{H}_d$, which will be an important tool in our analysis:  $$f(\mathbf{x},\mathbf{y},z) =\sum_{\lambda \in \widehat{\Z}_p } \sum_{ (\xi , \eta) \in \Q_p^{2d} / p^{\vartheta(\lambda)} \Z_p^{2d}} |\lambda|_p^d Tr[ \chi_{\xi , \eta} (\mathbf{x},\mathbf{y})\pi_\lambda (\mathbf{x},\mathbf{y},z) \widehat{f} (\xi , \eta , \lambda)],$$where $$ \widehat{f} (\xi , \eta , \lambda) \varphi = \mathcal{F}_{\mathbb{H}_d}[f](\xi , \eta ,\lambda) \varphi := \int_{\mathbb{H}_d} f(\mathbf{x},\mathbf{y},z) \overline{ \chi_{\xi , \eta}} (\mathbf{x},\mathbf{y}) \pi_\lambda^* (\mathbf{x},\mathbf{y},z) \varphi  \, dx \, dy \, dz, \,\, \, \, \, f \in L^2 (\mathbb{H}_d).$$This information is useful, but insufficient for many purposes, and we will like to have a more explicit description of $\widehat{\mathbb{H}}_d$. So, in this sub-section our goal is to provide for the reader explicit realizations of the representations $\pi_{(\xi , \eta , \lambda)}$ and their associated matrix coefficients. The job should be easy enough after the analysis in the past sub-section, specially considering the fact that the representation theory of the Heisenberg group is very well known. 

We will choose our representation space $\mathcal{H}_\lambda$ to be the sub-space of $L^2 (\Z_p^d)$ $$\mathcal{H}_\lambda := \mathrm{Span}_\C \{ \varphi_h \, : \, h \in \Z_p^d / p^{-\vartheta(\lambda)} \Z_p^d  \}, \, \, \, \varphi_h (u) := |\lambda|_p^{d/2} \1_{h + p^{-\vartheta(\lambda)} \Z_p^d} (u), \, \, \dim(\mathcal{H}_\lambda) =|\lambda|_p^d.$$And the representation will act on functions in this spaces via the formula $$\pi_{(\xi , \eta , \lambda)}(\mathbf{x},\mathbf{y},z) \varphi (u) := e^{2 \pi i \{\xi \cdot \mathbf{x} + \eta \cdot \mathbf{y} + \lambda (z + u \cdot \mathbf{y}) \}_p} \varphi (u + \mathbf{x}), \, \, \, \, \varphi \in \mathcal{H}_\lambda.$$
The operator $\pi_{(\xi , \eta , \lambda)}(\mathbf{x},\mathbf{y},z)$ defines an unitary operator on $L^2 (\Z_p^d)$, for any $(\mathbf{x},\mathbf{y},z) \in \mathbb{H}_d$, and its associated adjoint operator is given by $$\pi_{(\xi , \eta , \lambda)}^*(\mathbf{x},\mathbf{y},z) \varphi (u) := e^{-2 \pi i \{ \xi \cdot \mathbf{x} + \eta \cdot \mathbf{y} + \lambda (z + (u-\mathbf{x})\cdot \mathbf{y}) \}_p} \varphi (u - \mathbf{x}).$$
Notice how the space $\mathcal{H}_\lambda$, which is simply $\mathcal{D}_{n_\lambda}(\Z_p^d)$, $|\lambda|_p = p^{n_\lambda}$, the collection of smooth functions on $\Z_p^d$ such that $\varphi(u +v) = \varphi(u),$ for $ v \in p^{n_\lambda} \Z_p^d,$  is invariant under the action of $\pi_{(\xi , \eta , \lambda)}$. Using the natural basis for $\mathcal{H}_\lambda$, i.e.  $|\lambda|_p^{d/2} \1_{h+p^{-\vartheta(\lambda)}\Z_p^d}, \, \, h \in \Z_p^d/ p^{-\vartheta(\lambda)}\Z_p^d$, the associated matrix coefficients of the representations are given by \begin{align*}
    (\pi_{(\xi , \eta , \lambda)})_{h h'}&= (\pi_{(\xi , \eta , \lambda)} \varphi_h , \varphi_{h'})_{L^2 (\Z_p^d)} \\ &= |\lambda|_p^d \int_{\Z_p^d} e^{2 \pi i \{ \lambda(z + u \cdot \mathbf{y}) + ( \xi \cdot \mathbf{x} +  \eta \cdot \mathbf{y}) \}_p} \1_h (u +\mathbf{x}) \1_{h'} (u) du \\ &=|\lambda|_p^d  \1_{h - h'} (\mathbf{x}) \int_{h' + p^{-\vartheta(\lambda)} \Z_p^d} e^{2 \pi i \{ \lambda(z + u \cdot \mathbf{x}) + ( \xi \cdot \mathbf{x} + \eta \cdot \mathbf{y}) \}_p}  du \\ &= e^{ 2 \pi i \{ \lambda(z +  h' \cdot \mathbf{y} ) + (\xi \cdot \mathbf{x} + \eta \cdot \mathbf{y}) \}_p} \1_{h - h'}(\mathbf{x}) \Big( |\lambda|_p^d \int_{p^{-\vartheta(\lambda)} \Z_p^d} e^{2 \pi i \{ \lambda  u \cdot \mathbf{y}\}_p } du \Big) \\ &=e^{ 2 \pi i \{ \lambda(z + h' \cdot \mathbf{y} ) + (\xi \cdot \mathbf{x} + \eta \cdot \mathbf{y}) \}_p} \1_{h - h'}(\mathbf{x}) .
\end{align*}
With the same arguments,
\begin{align*}
    (\pi_{(\xi , \eta , \lambda)}^*)_{h h'}&= |\lambda|_p^d \int_{\Z_p^d} e^{-2 \pi i \{ \lambda(z + (u-\mathbf{x}) \cdot \mathbf{y}) + ( \xi \cdot \mathbf{x} + \eta \cdot \mathbf{y}) \}_p} \1_h (u -\mathbf{x}) \1_{h'} (u) du \\ &=|\lambda|_p^d  \1_{h' - h} (\mathbf{x}) \int_{h' + p^{-\vartheta(\lambda)} \Z_p^d} e^{-2 \pi i \{ \lambda(z + (u-\mathbf{x}) \cdot \mathbf{y}) + ( \xi \cdot \mathbf{x} + \eta \cdot \mathbf{y}) \}_p}  du \\ &= e^{ -2 \pi i \{ \lambda(z + (h'-\mathbf{x}) \cdot \mathbf{y} ) + ( \xi \cdot \mathbf{x} + \eta \cdot \mathbf{x}) \}_p} \1_{h'- h}(\mathbf{x}) \Big( |\lambda|_p^d \int_{p^{-\vartheta(\lambda)} \Z_p^d} e^{-2 \pi i \{ \lambda u y\}_p } du \Big) \\ &=e^{ -2 \pi i \{ \lambda(z +  (h'-\mathbf{x}) \cdot \mathbf{y} ) + (\xi \cdot \mathbf{x} + \eta \cdot \mathbf{y}) \}_p} \1_{h' - h}(x) \\ & = e^{ -2 \pi i \{ \lambda(z + h \cdot \mathbf{y} ) + (\xi \cdot \mathbf{x} + \eta \cdot \mathbf{y}) \}_p} \1_{h' - h}(\mathbf{x}) \\ &=\overline{(\pi_{(\xi , \eta , \lambda)})_{h' h}}.
\end{align*}
Now let us consider for a moment the sub-space $\mathcal{D}_n(\mathbb{H}_d)$ of $L^2 (\mathbb{H}_d)$ defined in Definition \ref{definotation}. It is easy to see that   $$\dim_\C (\mathcal{D}_n(\mathbb{H}_d)) = |\mathbb{H}_d ( \Z_p) / \mathbb{H}_d (p^n \Z_p)| = (p^n)^{2d +1},$$and for any $f \in \mathcal{D}_n(\mathbb{H}_d)$  and $\| (\xi , \eta , \lambda) \|_p>p^n$ we have \begin{align*}
    \widehat{f} &(\xi , \eta , \lambda):= \int_{\mathbb{H}_d} f(\mathbf{x},\mathbf{y},z) \overline{ \chi_{\xi , \eta}} (\mathbf{x},\mathbf{y}) \pi_\lambda^* (\mathbf{x},\mathbf{y},z)  \, d\mathbf{x} \, d\mathbf{y} \, dz \\ & = \int_{ \mathbb{H}_d ( \Z_p) / \mathbb{H}_d (p^n \Z_p)} \int_{(a,b,c)\mathbb{H}_d (p^n \Z_p)} f((a,b,c)(\mathbf{x},\mathbf{y},z)) \pi_{(\xi, \eta, \lambda)}^* ((a,b,c)(\mathbf{x},\mathbf{y},z))  d(\mathbf{x},\mathbf{y},z) \, d(a,b,c) \\ &= \sum_{(a, b ,c) \in \mathbb{H}_d ( \Z_p) / \mathbb{H}_d (p^n \Z_p)} f(a,b,c) \pi_{(\xi, \eta, \lambda)}^*(a,b,c) \int_{\mathbb{H}_d (p^n \Z_p)}\pi_{(\xi, \eta, \lambda)}^* (\mathbf{x},\mathbf{y},z)  d(\mathbf{x},\mathbf{y},z) = 0_{|\lambda|_p^d}.
\end{align*}
On the other hand,  we can also check how for the sub-space of $L^2 (\mathbb{H}_d)$ $$\mathcal{H}_n := Span_\C \{(\pi_{(\xi , \eta, \lambda)})_{h h'} \, : \, \|(\xi , \eta , \lambda) \|_p \leq p^n, \, \, \, h, h' \in \Z_p^d / p^{-\vartheta(\lambda)}\Z_p^d \},$$ it holds that$$\mathcal{H}_n \subseteq \mathcal{D}_n(\mathbb{H}_d), \, \, \, \text{and} \, \, \, \dim_\C(\mathcal{H}_n) = (p^n)^{2d+1}.$$In conclusion, we get that $\mathcal{D}_n(\mathbb{H}_d) = \mathcal{H}_n,$ and the functions $|\lambda|_p^{d/2}(\pi_{(\xi , \eta, \lambda)})_{h h'}$ form an orthonormal basis of $\mathcal{D}_n(\mathbb{H}_d)$, for any $n \in \N_0$. Since this is true for arbitrary $n$, we have just proven how the space of trigonometric polynomials on $\mathbb{H}_d$ coincides with the space of smooth functions, which in turns proves the density of the space of trigonometric polynomials on $L^r (\mathbb{H}_d)$, for $1 \leq r < \infty$. The results collected so far constitute the proof of Theorem \ref{TeoRepresentationsHd}.

\begin{rem}\label{remrelationFTHd}
Notice how any function on $\mathbb{H}_d$ can also be considered as a function on $L^2 (\Z_p^{2d+1})$. Then, given a function $f \in L^2 (\mathbb{H}_d)$, we can consider both, its group Fourier transform $\widehat{f}= \mathcal{F}_{\mathbb{H}_d} [f]$, and its $\Z_p$-Fourier transform $\mathcal{F}_{\Z_p^{2d+1}}[f]$. By using the Fourier inversion formula on $\Z_p^d$ we obtain the following relationship between  $\mathcal{F}_{\Z_p^{2d+1}}[f]$ and $\widehat{f}$: 
\begin{align*}
    \widehat{f}(\xi , \eta, \lambda) g (u) &= \int_{\mathbb{H}_d} f(\mathbf{x},\mathbf{y},z) \overline{ \chi_{\xi , \eta}} (\mathbf{x},\mathbf{y}) \pi_\lambda^* (\mathbf{x},\mathbf{y},z) g (u) \, d \mathbf{x} \, d \mathbf{y} \, dz\\ =\int_{\Z_p^{2d+1}}\sum_{(\alpha , \beta , \gamma) \in \widehat{\Z}_p^{2d+1} } &\mathcal{F}_{\Z_p^{2d+1}}[f] (\alpha , \beta , \gamma) e^{2 \pi i \{ (\alpha , \beta , \gamma) \cdot (\mathbf{x},\mathbf{y},z) - ( \xi \cdot \mathbf{x} + \eta \cdot \mathbf{y}) - \lambda(z + (u - \mathbf{x}) \cdot \mathbf{y})  \}_p}   g(u-\mathbf{x})  d \mathbf{x} \, d\mathbf{y} \, dz \\ &= \int_{\Z_p^{d}}\sum_{\alpha  \in \widehat{\Z}_p^{d} } \mathcal{F}_{\Z_p^{2d+1}}[f] (\alpha , \lambda (u - \mathbf{x}) + \eta ,  \lambda) e^{2 \pi i \{ (\alpha - \xi) \cdot \mathbf{x} \}_p}   g(u-\mathbf{x})  dx \\ &=\int_{\Z_p^{d}}\sum_{\alpha  \in \widehat{\Z}_p^{d} } \mathcal{F}_{\Z_p^{2d+1}}[f] (\alpha + \xi , \lambda  v + \eta,  \lambda) e^{2 \pi i \{ \alpha (u -v)\}_p}   g(v) dv \\ &= \int_{\Z_p^{d}}K_{f,(\xi , \eta , \lambda)} (u,v)   g(v)  dv, \, \, \, \, g \in L^2 (\Z_p^d), 
\end{align*}
where $$K_{f,(\xi , \eta , \lambda)} (u,v):= \sum_{\alpha  \in \widehat{\Z}_p^{d} } \mathcal{F}_{\Z_p^{2d+1}}[f] (\alpha + \xi , \lambda  v + \eta,  \lambda) e^{2 \pi i \{ \alpha (u -v)\}_p}.$$
In other words, we can think on the group Fourier transform of $f$ as an integral operator with kernel $K_{f,(\xi , \eta , \lambda)}$, acting on a certain finite-dimensional sub-space of $L^2 (\Z_p^d)$. This resembles the classical case of the Heisenberg group on the Euclidean space. In terms of the matrix coefficients we have: \begin{align*}
    \widehat{f}&(\xi , \eta, \lambda)_{h h'}\\  &=\int_{\Z_p^{2d+1}}\sum_{(\alpha , \beta , \gamma) \in \widehat{\Z}_p^{2d+1} } \mathcal{F}_{\Z_p^{2d+1}}[f] (\alpha , \beta , \gamma) e^{2 \pi i \{ (\alpha , \beta , \gamma) \cdot (\mathbf{x},\mathbf{y},z) - ( \xi \cdot \mathbf{x} + \eta \cdot \mathbf{y}) - \lambda(z + h \cdot \mathbf{y})  \}_p}   \1_{h'-h + p^{-\vartheta(\lambda)}\Z_p^d} (\mathbf{x})  d \mathbf{x} \, d\mathbf{y} \, dz \\ &= \int_{\Z_p^{d}}\sum_{\alpha  \in \widehat{\Z}_p^{d} } \mathcal{F}_{\Z_p^{2d+1}}[f] (\alpha , \lambda  h + \eta ,  \lambda) e^{2 \pi i \{ (\alpha - \xi) \cdot \mathbf{x}\}_p}   \1_{h'-h + p^{-\vartheta(\lambda)}\Z_p^d} (\mathbf{x})  d \mathbf{x} \\ &=\sum_{\alpha  \in \widehat{\Z}_p^{d} } \mathcal{F}_{\Z_p^{2d+1}}[f] (\alpha  , \lambda  h + \eta,  \lambda) \mathcal{F}_{\Z_p^d}[\1_{h'-h + p^{-\vartheta(\lambda)}\Z_p^d}](\xi - \alpha).    
\end{align*}
\end{rem}

\section{The Vladimirov Sub-Laplacian on the Heisenberg group}
A simple calculation proves how the matrix entries of the representations $\pi_{(\xi , \eta , \lambda)}$ are eigenfunctions of the VT operator of order $\alpha>0$, with corresponding eigenvalues $\| (\xi , \eta , \lambda) \|_p^\alpha - \frac{1 - p^{-(2d+1)}}{1 - p^{-(\alpha +2d+ 1)}}$. This shows how the VT operator is an example of a very special class of operators, which we call here globally hypoelliptic operators. For left invariant operators, its associated symbol in the sense of Definition \ref{defihypo} is independent of the variables $(x,y,z) \in \mathbb{H}_d$. We call these operators sometimes Fourier multipliers, and the VT operator together with the directional VT operators are important examples of such class. For Fourier multipliers, it is easy  to prove how the global hypoellipticity is completely determined by the behaviour at infinity of the symbol. 

\begin{rem}
Notice how, after describing explicitly the unitary dual of $\mathbb{H}_d$, a function in $f \in L^2 (\mathbb{H}_d)$ is a Schwartz function if and only if $$\| \widehat{f}(\xi, \eta, \lambda)\| \lesssim \| (\xi, \eta, \lambda) \|_p^{-k}, \, \, \, \, \text{for every} \, k \in \N_0.$$    
\end{rem}
\begin{lema}\label{Teohypoellipticity}
Let $T_\sigma$ be a Fourier multiplier. Then $T_\sigma$ is globally hypoelliptic if and only if there is an $m \in \mathbb{R}$ such that $\| (\xi, \eta ,\lambda) \|_p^m \lesssim \| \sigma (\xi, \eta, \xi) \|_{inf},$ for $\| ( \xi, \eta, \lambda) \|_p$ large enough. Here we are using the notation $$\| \sigma (\xi, \eta, \lambda) \|_{inf} := \inf \{\|\sigma(\xi, \eta, \lambda) v\| \, : \, \|v\|=1, \, \, v \in \mathcal{H}_{(\xi, \eta,\lambda)} \}.$$When $\| \sigma(\xi, \eta, \lambda)\|_{inf}$ is non-zero, $\sigma (\xi, \eta, \lambda)$ is invertible and $\| \sigma (\xi, \eta, \lambda)^{-1}\|_{op} = \| \sigma(\xi, \eta, \lambda)\|_{inf}^{-1}.$
\end{lema}
\begin{proof}
If $\| \sigma(\xi,\eta,\lambda)\|_{inf}^{-1} \lesssim \| (\xi,\eta,\lambda)\|_p^{-m}$ for large $\|(\xi,\eta,\lambda )\|_p$, then it is easy to see that $\widehat{f} (\xi,\eta,\lambda) = \sigma (\xi,\eta,\lambda)^{-1} \widehat{g}(\xi,\eta,\lambda)$, for large $\|(\xi,\eta,\lambda )\|_p$, so $f$ is a Schwartz function. For the reciprocal statement we proceed by contradiction. Assume that there is a sequence $\{(\xi_n,\eta_n,\lambda_n , m_n)  \}$ such that $m_n \to \infty,$  and $\| (\xi_n,\eta_n,\lambda_n)\|_p \to  \infty$ as $n \to \infty$, and $$ \| \sigma(\xi_n,\eta_n,\lambda_n)\|_{inf} \leq \| (\xi_n,\eta_n,\lambda_n) \|_p^{-m_n}.$$Also, define \[ \widehat{f} (\xi,\eta,\lambda) := \begin{cases}
    I_{d_{(\xi,\eta,\lambda)}} & \, \, \text{if} \, \, (\xi,\eta,\lambda) = (\xi_n,\eta_n,\lambda_n), \\ 0 & \, \, \text{in other case.} 
    \end{cases}\]Then $Tf \in \mathcal{S}(\mathbb{H}_d)$ but $f \in \mathcal{S}'(\mathbb{H}_d) \setminus \mathcal{S}(\mathbb{H}_d)$ contradicting the global hypoellipticity of $T_\sigma$. 

\end{proof}

The above proposition proves how the VT operator is a globally hypoelliptic operator, since its associated symbol $\sigma_{D^\alpha} (\xi , \eta , \lambda)$ is invertible and bounded below, for $\|(\xi , \eta , \lambda) \|_p$ large enough, though this is a trivial fact. For more general invariant operators, Lemma \ref{Teohypoellipticity} provides a simple condition in terms of the symbol, and we wish to exploit it to study some interesting operators on $\mathbb{H}_d$. We are particularly interested in the directional VT operators, and the polynomials in the directional VT operators. In general these operators are not simple to handle but, in some cases, one can manage to give some satisfactory information, and this will be the case for the following operators:

\begin{defi}\normalfont\label{defiHeiCompactSubLap}
Let $\mathfrak{h}_d = \mathrm{Span}_{\Z_p} \{X_1, ..., X_d, Y_1,..,Y_d,Z \}$ be the Heisenberg $\Z_p$-Lie algebra, and let us take any pair of collections of l.i. vectors $$\textbf{V}=\{V_1 ,..,V_d\} \subset \mathrm{Span}_{\Z_p} \{X_1, ..., X_d \}, \,\,\,\, \textbf{W}=\{W_1 ,..,W_d\} \subset \mathrm{Span}_{\Z_p} \{Y_1, ..., Y_d \}.$$Given $\alpha, \beta \in (\R^+)^d$, we define the \emph{Vladimirov sub-Laplacian associated to} $(\textbf{V}, \textbf{W}, \alpha , \beta)$, as the linear operator $$T_{\textbf{V}, \textbf{W}}^{\alpha , \beta} = \sum_{k=1}^d \partial_{V_k}^{\alpha_k} + \partial_{W_k}^{\beta_k}.$$Also, for $\gamma \in \R^+$ we define the Vladimirov Laplacian as  $$T_{\textbf{V}, \textbf{W}}^{\alpha , \beta, \gamma} = \sum_{k=1}^d \partial_{V_k}^{\alpha_k} + \partial_{W_k}^{\beta_k} + \partial_Z^\gamma .$$In particular, when $\alpha=\alpha_1 =...=\alpha_d=\beta_1 = ...=\beta_d >0,$ and $$\textbf{V}=\{X_1, ..., X_d \}, \, \, \, \,  \textbf{W}= \{Y_1, ..., Y_d \}, $$we will use the notations $$\mathscr{L}_{sub}^{\alpha}:= \sum_{k=1}^d \partial_{X_k}^{\alpha} + \partial_{Y_k}^{\alpha}, \, \, \, \mathscr{L}^{\alpha}:= \sum_{k=1}^d \partial_{X_k}^{\alpha} + \partial_{Y_k}^{\alpha} + \partial_{Z}^{\alpha}.$$    
\end{defi}

In order to study these operators, just like we did for for the Vladimirov-Taibleson operator, we want to calculate their associated symbols. Let us start with the symbols of the directional VT operators $\partial_{V_k}^{\alpha_k}, \partial_{W_k}^{\beta_k}$ and $ \partial_{Z}^{\alpha}$.
\begin{itemize}
    \item From the expression for the matrix coefficients, for $\partial_{W_k}^{\alpha}$ we have that its associated symbol $\sigma_{\partial_{W_k}^{\alpha}} (\xi, \eta, \lambda) = \partial_{W_k}^{\alpha} \pi_{(\xi, \eta , \lambda)}|_{(x,y,z)=(0,0,0)}$ is a diagonal matrix with entries  \[ \sigma_{\partial_{W_k}^{\alpha}} (\xi, \eta, \lambda)_{hh}  =\begin{cases}
 | W_k \cdot \eta |_p^\alpha - \frac{1 - p^{-1}}{1 - p^{-(\alpha +1)}}  \quad & \text{if} \quad | \lambda |_p = 1, \\
| W_k \cdot (\lambda  h + \eta) |_p^\alpha - \frac{1 - p^{-1}}{1 - p^{-(\alpha +1)}} \quad & \text{if} \quad | \lambda |_p >1.
\end{cases}
\]
    \item For $\partial_{V_k}^{\alpha}$ its associated symbol $\sigma_{\partial_{V_k}^{\alpha}} (\xi, \eta, \lambda) = \partial_{V_k}^{\alpha} \pi_{(\xi, \eta , \lambda)}|_{(x,y,z)=(0,0,0)}$ is a Toeplitz matrix with entries \[ \sigma_{\partial_{V_k}^{\alpha}} (\xi, \eta, \lambda)_{hh'}  =\begin{cases}
 | V_k \cdot \xi |_p^\alpha - \frac{1 - p^{-1}}{1 - p^{-(\alpha +1)}}  \quad & \text{if} \quad | \lambda |_p = 1, \\
\partial_{V_k}^\alpha (e^{2 \pi i \{\xi \cdot \}_p} \1_{h - h'})|_{u=0} \quad & \text{if} \quad | \lambda |_p >1.
\end{cases}
\]
\item The symbol of $\partial_{Z}^{\alpha}$ is the simplest, and it is given by the  matrix with entries \[ \sigma_{\partial_{Z}^{\alpha}} (\xi, \eta, \lambda)_{hh'}  =\begin{cases}
 0  \quad & \text{if} \quad | \lambda |_p = 1, \\
(|\lambda|_p^\alpha - \frac{1 - p^{-1}}{1 - p^{-(\alpha +1)}}) \delta_{h h'} \quad & \text{if} \quad | \lambda |_p >1.
\end{cases}
\]
\end{itemize}
In particular the standard basis of $\mathfrak{h}_d$ and for $\mathscr{L}^\alpha_{sub}$ we have:
\begin{itemize}
    \item For $\partial_{Y_k}^{\alpha}$ we have that its associated symbol $\sigma_{\partial_{Y_k}^{\alpha}} (\xi, \eta, \lambda) = \partial_{Y_k}^{\alpha} \pi_{(\xi, \eta , \lambda)}|_{(x,y,z)=(0,0,0)}$ is a diagonal matrix with entries  \[ \sigma_{\partial_{Y_k}^{\alpha}} (\xi, \eta, \lambda)_{hh}  =\begin{cases}
 | \eta_k |_p^\alpha - \frac{1 - p^{-1}}{1 - p^{-(\alpha +1)}}  \quad & \text{if} \quad | \lambda |_p = 1, \\
| \lambda  h_k + \eta_k |_p^\alpha - \frac{1 - p^{-1}}{1 - p^{-(\alpha +1)}} \quad & \text{if} \quad | \lambda |_p >1.
\end{cases}
\]
    \item For $\partial_{X_k}^{\alpha}$ its associated symbol $\sigma_{\partial_{X_k}^{\alpha}} (\xi, \eta, \lambda) = \partial_{X_k}^{\alpha} \pi_{(\xi, \eta , \lambda)}|_{(x,y,z)=(0,0,0)}$ is a Toeplitz matrix with entries \[ \sigma_{\partial_{X_k}^{\alpha}} (\xi, \eta, \lambda)_{hh'}  =\begin{cases}
 | \xi_k |_p^\alpha - \frac{1 - p^{-1}}{1 - p^{-(\alpha +1)}}  \quad & \text{if} \quad | \lambda |_p = 1, \\
\partial_{u_k}^\alpha (e^{2 \pi i \{\xi \cdot \}_p} \1_{h - h'})|_{u=0} \quad & \text{if} \quad | \lambda |_p >1.
\end{cases}
\]
\end{itemize}

Summing up, the symbol of the Vladimirov Sub-Laplacian $T_{\textbf{V}, \textbf{W}}^{\alpha , \beta}$ can be written as
\[ \sigma_{T_{\textbf{V}, \textbf{W}}^{\alpha , \beta}} (\xi, \eta, \lambda)_{hh'}  =\begin{cases}
 \sum_{k=1}^d | V_k \cdot \xi |_p^{\alpha_k} +| W_k \cdot  \eta |_p^{\beta_k} - 2d\frac{1 - p^{-1}}{1 - p^{-(\alpha +1)}}  \quad & \text{if} \quad | \lambda |_p = 1, \\
\big(  \sum_{k=1}^d \partial_{V_k}^{\alpha_k} + |W_k \cdot ( \lambda  h + \eta) |_p^{\beta_k} - d\frac{1 - p^{-1}}{1 - p^{-(\alpha +1)}} \big) (e^{2 \pi i \{\xi \cdot \}_p} \1_{h - h'})|_{u=0} \quad & \text{if} \quad | \lambda |_p >1.
\end{cases}
\]
In particular for $\mathscr{L}^\alpha_{sub}$:
\[ \sigma_{\mathscr{L}^\alpha_{sub}} (\xi, \eta, \lambda)_{hh'}  =\begin{cases}
 \sum_{k=1}^d | \xi_k |_p^\alpha +| \eta_k |_p^\alpha - 2d\frac{1 - p^{-1}}{1 - p^{-(\alpha +1)}}  \quad & \text{if} \quad | \lambda |_p = 1, \\
\big(  \sum_{k=1}^d \partial_{u_k}^\alpha + | \lambda  h_k + \eta_k |_p^\alpha - d\frac{1 - p^{-1}}{1 - p^{-(\alpha +1)}} \big) (e^{2 \pi i \{\xi \cdot \}_p} \1_{h - h'})|_{u=0} \quad & \text{if} \quad | \lambda |_p >1.
\end{cases}
\]
It is clear that $$\| (\xi , \eta ,\lambda) \|_p^{\min\{ \alpha_k, \beta_k\}}\lesssim |\sigma_{T_{\textbf{V}, \textbf{W}}^{\alpha , \beta}} (\xi, \eta, \lambda)|  \lesssim \| (\xi , \eta ,\lambda) \|_p^{\max\{ \alpha_k, \beta_k\}},$$and  $|\sigma_{\mathscr{L}^\alpha_{sub}} (\xi, \eta, \lambda)| \asymp \|(\xi , \eta , \lambda)\|_p^\alpha$ when $|\lambda|_p =1$. When $|\lambda|_p>1$, in order to estimate $\|\sigma_{T_{\textbf{V}, \textbf{W}}^{\alpha , \beta}} (\xi, \eta, \lambda)\|_{op}$ and $\|\sigma_{T_{\textbf{V}, \textbf{W}}^{\alpha , \beta}} (\xi, \eta, \lambda)\|_{inf}$, we only need to consider the action of $T_{\textbf{V}, \textbf{W}}^{\alpha , \beta}$ on each finite-dimensional sub-space $$\mathcal{V}_{(\xi , \eta, \lambda)}:= Span_\C \{ (\pi_{(\xi , \eta , \lambda)})_{hh'} \, : \, h,h' \in \Z_p^d / p^{-\vartheta(\lambda)} \Z_p^d \}.$$Notice how, for any function $f_h \in \mathcal{D}(\mathbb{H}_d)$ with the form $$f_{h'}(x,y,z) = |\lambda|_p^d \sum_{h \in \Z_p^d / p^{-\vartheta(\lambda)} \Z_p^d} \pi_{(\xi , \eta , \lambda)}(x,y,z)_{hh'} \widehat{f}(\xi , \eta , \lambda)_{h'h},$$ we have $$T_{\textbf{V}, \textbf{W}}^{\alpha , \beta} f_{h'} (x,y,z) = \big( \sum_{k=1}^d \partial_{V_k}^{\alpha_k} + |W_k \cdot ( \lambda  h' + \eta ) |_p^{\beta_k} - d\frac{1 - p^{-1}}{1 - p^{-(\alpha +1)}} \big)f_{h'}(x,y,z).$$
Let us define the sub-spaces 
$$\mathcal{V}_{(\xi , \eta, \lambda)}^{h'}:= Span_\C \{ (\pi_{(\xi , \eta , \lambda)})_{hh'} \, : \, h \in \Z_p^d / p^{-\vartheta(\lambda)} \Z_p^d \},$$so that $$\mathcal{V}_{(\xi , \eta, \lambda)} = \bigoplus_{h' \in  \Z_p^d / p^{-\vartheta(\lambda)} \Z_p^d} \mathcal{V}_{(\xi , \eta, \lambda)}^{h'} ,$$and we know that $T_{\textbf{V}, \textbf{W}}^{\alpha , \beta}$ acts on $\mathcal{V}_{(\xi , \eta, \lambda)}^{h'}$ as the operator $$\sum_{k=1}^d \partial_{V_k}^{\alpha_k} + |W_k \cdot ( \lambda  h + \eta ) |_p^{\beta_k} - d\frac{1 - p^{-1}}{1 - p^{-(\alpha +1)}}.$$
The space $\mathcal{V}^{h'}_{(\xi, \eta, \lambda)}$ is actually equal to the space $$e^{2 \pi i \{(\xi x + \eta y) + \lambda(z +  h' y) \}_p} \mathcal{D}_{n_\lambda}(\Z_p^d) = \{ e^{2 \pi i \{(\xi x + \eta y) + \lambda(z +  h' y) \}_p} g \, : \, g \in \mathcal{D}_{n_\lambda} (\Z_p^d)\},$$ so that, an alternative basis for the space $\mathcal{V}_{(\xi , \eta , \lambda)}^{h}$ could be $$\mathscr{e}_{(\xi, \eta, \lambda, \tau), } (x) = e^{2 \pi i \{((\tau+\xi) x + \eta y) + \lambda(z +  h' y) \}_p}, \, \, \, 1 \leq \| \tau \|_p \leq |\lambda|_p.$$This is actually a basis of eigenfunctions for $\sum_{k=1}^d \partial_{V_k}^{\alpha_k} + |W_k \cdot ( \lambda  h' + \eta ) |_p^{\beta_k} - d\frac{1 - p^{-1}}{1 - p^{-(\alpha +1)}},$ and the corresponding eigenvalues are given by $$\sum_{k=1}^d |V_k \cdot(\tau + \xi)|_p^{\alpha_k} + |W_k \cdot ( \lambda  h' + \eta ) |_p^{\beta_k} - 2d\frac{1 - p^{-1}}{1 - p^{-(\alpha +1)}},$$so that by taking the union over all $(\xi, \eta, \lambda) \in \widehat{\mathbb{H}}_d$, $h' \in \Z_p^d / p^{-\vartheta(\lambda)}\Z_p^d,$ $1 \leq \| \tau \|_p \leq |\lambda_p|,$ we obtain the full spectrum of $T_{\textbf{V}, \textbf{W}}^{\alpha , \beta}$.

Finally, in order to prove the global hypoellipticity of $T_{\textbf{V}, \textbf{W}}^{\alpha , \beta}$, it should be clear that \begin{align*}
    \| \sigma_{T_{\textbf{V}, \textbf{W}}^{\alpha , \beta}}(\xi , \eta, \lambda) \|_{op} &= \| T_{\textbf{V}, \textbf{W}}^{\alpha , \beta} |_{\mathcal{V}_{(\xi , \eta, \lambda)}} \|_{op}  = \max_{h'  \in \Z_p^d / p^{-\vartheta(\lambda)} \Z_p^d }  \| T_{\textbf{V}, \textbf{W}}^{\alpha , \beta} |_{\mathcal{V}_{(\xi , \eta, \lambda)}^{h'}} \|_{op},
\end{align*}
and $$\| \sigma_{T_{\textbf{V}, \textbf{W}}^{\alpha , \beta}}(\xi , \eta, \lambda) \|_{inf} = \| T_{\textbf{V}, \textbf{W}}^{\alpha , \beta} |_{\mathcal{V}_{(\xi , \eta, \lambda)}} \|_{inf}  = \min_{h'  \in \Z_p^d / p^{-\vartheta(\lambda)} \Z_p^d }  \| T_{\textbf{V}, \textbf{W}}^{\alpha , \beta} |_{\mathcal{V}_{(\xi , \eta, \lambda)}^{h'}} \|_{inf},$$so, we can conclude that \begin{align*}
    \| &\sigma_{T_{\textbf{V}, \textbf{W}}^{\alpha , \beta}}(\xi , \eta, \lambda) \|_{op} \\ & \leq \max_{h ' \in \Z_p^d / p^{-\vartheta(\lambda)} \Z_p^d }(\sum_{k=1}^d \max \{|\lambda_p|^{\alpha_k} , |V_k \cdot \xi|_p^{\alpha_k} \}   + | W_k \cdot (\lambda  h' + \eta) |_p^{\beta_k} - 2d\frac{1 - p^{-1}}{1 - p^{-(\alpha +1)}}) \\ & \lesssim \max \{|\lambda|_p^{\max\{\alpha_k\}}, \| \xi \|_p^{\max\{\alpha_k\}}, \| \lambda  h' + \eta \|_p^{\max\{\beta_k\}} \} \leq \|(\xi , \eta, \lambda) \|_p^{\max\{\alpha_k, \beta_k\}}.
\end{align*}
Similarly
\begin{align*}
    \| \sigma_{T_{\textbf{V}, \textbf{W}}^{\alpha , \beta}}(\xi , \eta, \lambda) \|_{inf} &\geq \min_{h'  \in \Z_p^d / p^{\mathfrak{o}(\lambda_k)} \Z_p^d } \| \xi  \|_p^{\alpha_k} + \sum_{k=1}^d | W_k \cdot ( \lambda  h' + \eta) |_p^{\beta_k} - \frac{1 - p^{-1}}{1 - p^{-(\alpha +1)}}  \\ & \gtrsim \min_{h \in \Z_p^d / p^{-\vartheta(\lambda)}\Z_p^d} \max \{ \| \xi  \|_p^{\min\{\alpha_k\}} , \| \lambda  h' + \eta \|_p^{\min\{\beta_k\}} \},
\end{align*}
which also implies in this way the estimate $$\min_{h' \in \Z_p^d / p^{-\vartheta(\lambda)}\Z_p^d} \max \{ \| \xi  \|_p^{\alpha} , \| \lambda  h' + \eta \|_p^{\alpha} \} \lesssim \| \sigma_{\mathscr{L}_{sub}^\alpha} (\xi , \eta, \lambda)\|_{inf} \leq \| \sigma_{\mathscr{L}_{sub}^\alpha} (\xi , \eta, \lambda)\|_{op} \lesssim \|(\xi , \eta, \lambda) \|_p^\alpha,   $$  
proving in this way the global hypoellipticity of $T_{\textbf{V}, \textbf{W}}^{\alpha , \beta}$ and $\mathscr{L}_{sub}^\alpha$. Actually, $T_{\textbf{V}, \textbf{W}}^{\alpha , \beta}$ and $\mathscr{L}^\alpha_{sub}$ belong to an important class of hypoelliptic operators, which we call here \emph{sub-elliptic operators}.

\begin{defi}\label{defiellipticandsubellipticsymbol}\normalfont
Let $\sigma$ be a symbol, and let $T_\sigma$ be its associated pseudo-differential operator.
\begin{itemize}
    \item We say that $\sigma$ is an \emph{elliptic symbol of order} $m \in \R$, if there are $C_1, C_2 >0$ and $m \in \R$ such that $$C_1 \| (\xi, \eta, \lambda) \|_p^m \leq \| \sigma((x,y,z) ,(\xi, \eta, \lambda)) \|_{inf} \leq  \| \sigma((x,y,z) ,(\xi, \eta, \lambda)) \|_{op} \leq C_2 \| (\xi, \eta, \lambda) \|_p^m .$$If there is an $0<\delta\leq m$ such that $$C_1 \| (\xi, \eta, \lambda) \|_p^{m-\delta} \leq \| \sigma((x,y,z) ,(\xi, \eta, \lambda)) \|_{inf} \leq  \| \sigma((x,y,z) ,(\xi, \eta, \lambda)) \|_{op} \leq C_2 \| (\xi, \eta, \lambda) \|_p^m ,$$ we say that $\sigma$ is a \emph{sub-elliptic symbol}.  
    \item We say that $T_\sigma \in \mathcal{L}(H_2^{s+m}(\mathbb{H}_d) , H_2^{s}(\mathbb{H}_d))$ is a \emph{globally elliptic operator of order} $m$, if $T_\sigma f \in H_2^{s}(\mathbb{H}_d) $ implies $f \in H_2^{s+m}(\mathbb{H}_d)$. We say that $T_\sigma \in \mathcal{L}(H_2^{s+m}(\mathbb{H}_d) , H_2^{s}(\mathbb{H}_d))$ is a \emph{globally sub-elliptic operator of order} $m$, if $T_\sigma f \in H_2^{s}(\mathbb{H}_d) $ implies $f \in H_2^{s+m-\delta}(\mathbb{H}_d)$, for some $0<\delta\leq m$. Here the $L^2$-based Sobolev space $H_2^{s}(\mathbb{H}_d)$, $s \in \R$, is defined as $$H_2^{s}(\mathbb{H}_d)=\{f \in L^2 (\mathbb{H}_d) \, : \, \mathbb{D}^s f \in L^2 (\mathbb{H}_d)\}.$$Equivalently, $f \in H_2^{s}(\mathbb{H}_d)$ if and only if $$\| f\|_{H_2^{s}(\mathbb{H}_d)}  = \Big( \sum_{(\xi, \eta, \lambda) \in \widehat{\mathbb{H}}_d} |\lambda|_p^d  \| (\xi , \eta, \lambda)\|_p^{2s} \| \widehat{f}(\xi , \eta, \lambda) \|_{HS}^2 \Big)^{1/2}<\infty.$$
\end{itemize}      
\end{defi}
\begin{rem}
    Notice how for invariant operators it is true that $T_\sigma$ is elliptic, or sub-elliptic, if and only if its symbol $\sigma$ is elliptic or sub-elliptic, respectively. Also, in both cases $T_\sigma$ defines a globally hypoelliptic operator.
\end{rem}
In the case of the Vladimirov Laplacian $T_{\textbf{V}, \textbf{W}}^{\alpha , \beta, \gamma} ,$ by taking $$\alpha_1=...=\alpha_d = \beta_1 = ... = \beta_d = \gamma = s>0,$$the estimates throughout this section show that
$$ \| (\xi, \eta, \lambda) \|_p^s \lesssim \| \sigma_{T_{\textbf{V}, \textbf{W}, Z}^{s}}(\xi, \eta, \lambda) \|_{inf} \leq  \| \sigma_{T_{\textbf{V}, \textbf{W},Z}^{s }}(\xi, \eta, \lambda) \|_{op} \lesssim  \| (\xi, \eta, \lambda) \|_p^s ,$$ because $\| (\xi , \eta, \lambda)\|_p^s \lesssim \max \{ \| \xi  \|_p^s , \| \lambda  h' + \eta \|_p^s \} + |\lambda|_p^s$, where $$T_{\textbf{V}, \textbf{W},Z}^{\alpha } : = \sum_{k=1}^d \partial_{V_k}^{\alpha} + \partial_{W_k}^{\alpha} + \partial_Z^\alpha .$$
This means that the Vladimirov Laplacian is an elliptic operator, and therefore an alternative description of the Sobolev space $H_2^{s}(\mathbb{H}_d)$ is $$H_2^{s}(\mathbb{H}_d) = \{f \in L^2(\mathbb{H}_d) \, : \, T_{\textbf{V}, \textbf{W},Z}^{\alpha }f \in L^2 (\mathbb{H}_d)  \}.$$ 

The results collected so far constitute the statement of Theorem \ref{TeoSpectrumSublaplacianHd}, whose proof is now completed.
\section{Relation with non-invariant vector fields on $\Z_p$ }
To conclude this paper, let us remark how the directional VT operator $\partial_{Y_i}^\alpha$ actually acts on functions $f \in L^2(\mathbb{H}_d) \cong L^2 (\Z_p^{2d+1})$ the same way as the directional VT operator in the direction of the vector field $$V_i (\mathbf{x}, \mathbf{y}, z)= e_{d+i} +  x_i e_{2d+1}, $$via the formula \begin{align*}
    \partial_{V_i}^\alpha f(\mathbf{x}, \mathbf{y},z) &= \int_{\Z_p} \frac{f((\mathbf{x}, \mathbf{y},z) + t V_i(\mathbf{x}, \mathbf{y},z)) - f(\mathbf{x}, \mathbf{y},z)}{|t|_p^{\alpha + 1}}dt \\&= \int_{\Z_p} \frac{f(\mathbf{x}, \mathbf{y} + t Y_{i},z + t x_i ) - f(\mathbf{x}, \mathbf{y},z)}{|t|_p^{\alpha + 1}}dt = \partial_{Y_i}^\alpha f(\mathbf{x} , \mathbf{y}, z). 
\end{align*}However, from this perspective the resulting operator is non-invariant, and its associated symbol, corresponding to the Fourier transform on $\Z_p^{2d+1}$ is going to be $$\sigma_{\partial_{V_i}^\alpha} (\gamma)= \int_{\Z_p} \frac{e^{2 \pi i \{ \gamma \cdot t (e_{d+i} +  x_i e_{2d+1})  \}_p } - 1}{|t|_p^{\alpha + 1}}dt=|\gamma \cdot (e_{d+i} +  x_i e_{2d+1})|_p^\alpha - \frac{1 - p^{-1}}{1 - p^{-(\alpha + 1)}}, \quad \gamma \in \Q_p^{2d + 1} / \Z_p^{2d+1}.$$So, the Vladimirov sub-Laplacian $\mathcal{L}_{sub}^\alpha$ acts on $\Z_p^{2d+ 1}$ like the, in principle non-invariant, operator $$\sum_{i=1}^d \partial_{u_i}^\alpha + \partial_{V_i}^\alpha, $$for which is possible to prove the global hypoellipticity on $\Z_p^{2d+1}$. However, what we showed here is how the group Fourier transform on $\mathbb{H}_d$ is much better to study this operator, because the operator becomes invariant, and tha ultimately allowed us to find a basis of eigenfunctions, which would actually be non trivial if we were working on $\Z_p^{2d +1}$ with the commutative structure.     
  
\nocite{*}
\bibliographystyle{acm}
\bibliography{main}
\Addresses

\end{document}